\documentclass[11pt]{article}

\usepackage{mathrsfs,bbm,bm,enumitem, diagbox}
\usepackage[margin=1in]{geometry}

\renewcommand{\P}{\mathbb{P}}
\newcommand{\E}{\mathbb{E}}
\newcommand{\Var}{\mathrm{Var}}
\newcommand{\Cov}{\mathrm{Cov}}
\newcommand{\cum}{\mathrm{cum}}

\newcommand{\detm}{\mathrm{det}}

\newcommand{\ind}{\mathbbm{1}}
\newcommand{\les}{\lesssim}
\newcommand{\ef}[1]{F_{#1}^{\e}}
\newcommand{\efla}[2]{F_{#1,#2}^\la}
\newcommand{\tr}{\mathrm{tr}}

\newcommand{\norm}[1]{\left\|#1\right\|}

\usepackage{amsmath,graphicx, amsthm, amssymb, epsfig}
\usepackage{xcolor}
\usepackage[normalem]{ulem}
\numberwithin{equation}{section}

\usepackage{hyperref}
\hypersetup{
    colorlinks=true,
    linkcolor=blue,
    filecolor=magenta,      
    urlcolor=cyan,
    citecolor=teal,
    pdfpagemode=FullScreen,
}

\newcommand{\e}{\epsilon}

\newcommand{\de}{\delta}
\newcommand{\br}{\mathbb{R}}

\newcommand{\pa}{\partial}
\newcommand{\bt}{\beta}
\newcommand{\al}{\alpha}
\newcommand{\la}{\lambda}

\newcommand{\be}{\begin{equation}}
\newcommand{\ee}{\end{equation}}
\def\bs#1\es{
    \begin{equation}\begin{split}
    #1
    \end{split}\end{equation}
}
\def\bsn#1\esn{
    \begin{equation*}\begin{split}
    #1
    \end{split}\end{equation*}
}

\newcommand{\dd}{\mathrm{d}}
\newcommand{\op}{\mathrm{op}}

\newcommand{\bma}{\begin{pmatrix}}
\newcommand{\ema}{\end{pmatrix}}

\newcommand{\us}{\mathcal U}

\newcommand{\CF}{\mathcal F}

\newcommand{\argc}[1]{c_{#1}}

\newcommand{\CH}{\mathcal{H}}

\newcommand{\CO}{\mathcal{O}}

\newtheorem{theorem}{Theorem}[section]
\newtheorem{lemma}[theorem]{Lemma}
\newtheorem{corollary}[theorem]{Corollary}

\theoremstyle{definition}
\newtheorem{assumption}[theorem]{Assumption}
\newtheorem{definition}[theorem]{Definition}
\newtheorem{remark}[theorem]{Remark}

\begin{document}

\title{High-dimensional Laplace asymptotics up to the concentration threshold}
\author{Alexander Katsevich\\
School of Data, Mathematical, and Statistical Sciences,\\ University of Central Florida\\ \texttt{alexander.katsevich@ucf.edu}
\and 
Anya Katsevich\\
Department of Statistical Science, Duke University\\ 
\texttt{anya.katsevich@duke.edu}
}

\date{}
\maketitle

\begin{abstract}
We study high-dimensional Laplace-type integrals of the form
$$
I(\la):=\left(\frac{\la}{2\pi}\right)^{d/2}\int_{\br^d} g(x)e^{-\la f(x)}\dd x,
$$
in the regime where $d$ and $\la$ are both large. Until now, rigorous bounds for the Laplace expansion in growing dimension have been restricted to the ``Gaussian-approximation'' regime, known to hold when $d^2/\la\to0$. This excludes many practically relevant regimes, including those arising in physics and modern high-dimensional statistics, which operate beyond this threshold while still satisfying the concentration condition $d/\la\to0$. Here, we close this gap. We develop an explicit asymptotic expansion for $\log I(\la)$ with quantitative remainder bounds that remain valid throughout this intermediate region, arbitrarily close to the concentration threshold $d/\la\to0$.

Fix any $L\ge1$ and suppose $g(0)=1$. Assume that, in a neighborhood of the minimizer of $f$, the operator norms of the derivatives of $f$ and $g$ are bounded independently of $d$ and $\la$ through orders $2(L+1)$ and $2L$, respectively. Assuming also some mild global growth conditions on $f$ and $g$, we prove that
\be\label{I-abstract}
\log I(\la)=\sum_{k=1}^{L-1} b_k(f,g)\la^{-k}+\CO(d^{L+1}/\la^L),\qquad d^{L+1}/\la^L\to0,
\ee
and that the coefficients satisfy $b_k(f,g)=\CO(d^{k+1})$. Moreover, the coefficients $b_k(f,g)$ coincide with those arising from the formal cumulant-based expansion of $\log I(\la)$.

In addition, we study the problem of computing expectations against, and sampling from, concentrating Laplace-type probability densities $\pi(x)\propto e^{-\la f(x)}$. For computing expectations of smooth observables $g$, we propose an approximation based on~\eqref{I-abstract}. For sampling, we construct a family of push-forward densities $\hat\pi_L:=(x_L)_\# \mathcal N(0,\la^{-1}I_d)$, $L=1,2,3\dots$ approximating $\pi$ with accuracy $\operatorname{TV}(\pi,\hat\pi_L)\lesssim d^{L+1}/\la^L$. Here, the maps $x_L$ are explicit polynomials. By taking $L$ large enough, here too, we can take $d$ arbitrarily close to the concentration threshold $d=o(\la)$. 


\end{abstract}

\section{Introduction}

Laplace-type integrals of the form
\be\label{lap}
I(\la):=\left(\frac{\la}{2\pi}\right)^{d/2}\int_{\mathbb{R}^d} g(x)e^{-\la f(x)}\dd x
\ee
are a fundamental object in asymptotic analysis and a ubiquitous tool for deriving tractable approximations to otherwise intractable quantities, such as normalizing constants and expectations. When $\la$ is large, these integrals are dominated by neighborhoods of minimizers of $f$, and Laplace asymptotics yields explicit leading-order formulas and systematic higher-order corrections. Such approximations underpin both rigorous analysis (e.g. sharp tail probabilities and free-energy expansions) and practical computation in settings where direct numerical integration is infeasible.

In the classical regime where the dimension $d$ is fixed and $\la\to\infty$, Laplace’s method and its higher-order refinements are well developed \cite[Section 8.3]{bh}, \cite[Section 5, Chapter IX]{wong}. At the opposite extreme, there is also an infinite-dimensional theory (e.g. on Wiener space) in which $\la^{-1}$ plays the role of a small-noise parameter~\cite{arous1988}. What remains comparatively less understood is the intermediate regime in which the dimension grows with the large parameter. This ``growing-$d$'' regime is now very common in modern statistics, where Laplace-type integrals (or ratios thereof) appear as marginal likelihoods and posterior expectations in models whose parameter dimension increases with sample size. See Section~\ref{ssec:statistics_motivation} for more details. Likewise, the absence of rigorously justified ``growing-$d$'' expansions created a significant void in areas such as statistical physics, Euclidean quantum field theories (QFT), and chemistry, where formal expansions have been in use for a long time without adequate justification, see Section~\ref{ssec:physics_motivation}.

Recent work by the second author establishes a high-dimensional Laplace expansion (LE) of $I(\la)$, with explicit remainder control, under the scaling $d^2/\la\to 0$~\cite{katsevich2025a}. The proof centers around an approximation of $f$ in the exponent by its second-order Taylor expansion. Thus the expansion of $I(\la)$ is closely tied to the problem of approximating concentrating densities, of the form $\pi\propto e^{-\la f}$, by Gaussian distributions. In fact, the condition $d^2\ll\la$ arises in numerous contexts as a threshold for Gaussian approximation accuracy. In the context of Gaussian approximation to posteriors in Bayesian inference, $d^2\ll\la$ appears in the works~\cite{katsevich2024b, katsevich2024c, kasprzak2025, katsevich2025unified}. The condition $d^2\ll\la$ also arises in proofs of high-dimensional Central Limit Theorems (CLTs); see~\cite{portnoy1986, zhilova2022, katsevich2025b}. In the CLT, $\la$ (usually denoted $n$) is the given number of i.i.d. $d$-dimensional random vectors, whose average is of interest. In fact, the aforementioned works~\cite{katsevich2025a, katsevich2024b,portnoy1986} establish lower bounds as well, proving $d^2/\la\ll1$ is necessary for the Gaussian approximations to be accurate. Although some works have in fact proved the CLT under much weaker conditions on $d$ relative to $\la$, the Gaussian approximation holds in a much weaker sense. See Section 3.1 and Remark 2 of~\cite{chernozhukov2023high} for an overview of this line of work.

Thus to summarize, the regime $d^2/\la\to0$ arises as a critical threshold in numerous results on approximating $d$-dimensional densities or integrals involving $e^{-\la f}$. Another critical threshold is the regime $d/\la\to0$, which is generically required for distributions $\propto e^{-\la f}$ to concentrate near the minimizer of $f$; see e.g.~\cite[Corollary 2.1]{ghosal2000} and~\cite[Section 4.1]{spok23}. Clearly, if there is no concentration around the minimizer of $f$, one should not expect an expansion to hold whose terms are determined solely by the derivatives of $g$ and $f$ at this single point.

This leaves an intriguing intermediate region where $d^2/\la$ does not vanish or even diverges to infinity, while $d/\la$ still converges to zero. Due to concentration about the minimizer of $f$, one expects that $I(\la)$ in~\eqref{lap} can still be characterized by derivatives of $f$ and $g$ at the minimizer. But it is unclear how to obtain a closed form approximation to the integral without leaning on the tractable Gaussian integral obtained by replacing $f$ with its second-order Taylor expansion. To our knowledge, there has been no rigorously justified explicit Laplace-type approximation of $I(\la)$ in this intermediate regime.


\emph{In this paper we have achieved an important milestone by completely characterizing this largely unexplored intermediate regime under natural local regularity and global growth conditions.} We derive an explicit asymptotic series approximating \eqref{lap} such that for each expansion order $L\geq1$, the remainder is negligible as long as $d^{L+1}/\la^L\to0$. 

The series approximates the \emph{logarithm} of the integral $I(\la)$, and as we explain below, this is precisely what allows us to push $d$ above the $\sqrt\la$ barrier.
Let $x_\star$ be the global minimizer of $f$ and assume without loss of generality that $\nabla^2f(x_\star)=I_d$. Assuming bounded operator norms of derivatives of $f$ and $g$ near $x_\star$ through orders $2L+2$ and $2L$, respectively, and mild global growth conditions on $f$ and $g$, we show that
\bs\label{main res intro}
\log I(\la)=\sum_{k=1}^{L-1} b_k(f,g)\la^{-k}+\CO(d^{L+1}/\la^L)
\es
for some coefficients $b_k(f,g)$ that satisfy $b_k(f,g)=\CO(d^{k+1})$. 
If $d\les \log\la$, then additional factors of $\log\la$ are present in the remainder.
The sum with respect to $k$ is omitted if $L=1$, and we assumed without loss of generality that $g=1$ at $x_\star$. The coefficients $b_k$ are already well-known, and can be derived from formal cumulant expansions~\cite[Section 3.10]{mccullagh2018}.

In fact, more broadly, our contribution is a powerful technique to tackle Laplace integrals. We use this technique not only to prove~\eqref{main res intro}, but also to solve a related problem of approximating a Laplace-type \emph{probability} density, of the form $\pi(x)\propto e^{-\la f(x)}$. Namely, we construct a transformation $x_L:\br^d\to\br^d$ such that $\hat\pi_L:=(x_L)_{\#}\mathcal N(0,\la^{-1}I_d)$ approximates $\pi$:
\be\label{TV-intro}
\mathrm{TV}(\pi,\hat\pi_L )\les d^{L+1}/\la^L.
\ee Here, $\mathcal N(0,\la^{-1}I_d)$ is the Gaussian distribution with  zero mean and covariance matrix $\la^{-1}I_d$, while TV stands for the total variation distance. The subscript $\#$ in the definition of $\hat\pi_L$ denotes the push-forward. This means $\hat\pi_L$ is the law of the random variable $x_L(Z)$ when $Z$ is distributed as $\mathcal N(0,\la^{-1}I_d)$. The reason $\hat\pi_L$ is so useful is that it is easy to sample from, precisely due to this push-forward construction. Cheaply generating approximate samples from an untractable density $\pi\propto e^{-\la f}$ is an important problem in Bayesian statistics. See Section~\ref{ssec:statistics_motivation} for more details.

\paragraph{Significance.}
A central significance of our result is that it advances the modern asymptotic analysis program of developing expansions of ubiquitous Laplace type integrals that remain accurate in increasingly high-dimensional regimes. The best known results to date work under the assumption $d^2/\la\to0$. In contrast, we extend the range of dimensions for which one can make precise asymptotic statements with explicit formulas until the very limit, because beyond $d/\la\to0$ there is no concentration any longer. \emph{In this sense, our work essentially completes the classical Laplace program for concentrating finite-dimensional integrals under natural smoothness and growth assumptions, by proving remainder bounds on the high-dimensional LE up to the concentration threshold.} We show that a constructive analytic approximation remains valid whenever $d^{L+1}/\la^L \to 0$ for any $L\ge 1$, even in the genuinely intermediate region where $d^2/\la \not\to 0$ (indeed, where $d^2/\la$ may diverge), by identifying an explicit exponential correction at the level of the log-integral. 

The implications of our results across many areas of science and engineering are numerous. Two particularly noteworthy applications deserve special mention. 

The first application is in physics, including statistical physics and QFT, where Laplace-type integrals encode quantities of central importance. In the high-dimensional, many-degrees-of-freedom regimes relevant to these fields, such quantities are often evaluated via formal LEs, typically without appropriate remainder bounds; see Section~\ref{ssec:physics_motivation} for details. This theoretical gap, which dates back at least to the Darwin--Fowler steepest-descent approach in 1922~\cite{darwin1922}, has long limited the rigor of many computations. \emph{Our results fill the century-old gap} by placing these Laplace calculations on firm mathematical footing for a broad class of finite-dimensional large-system models.

Another area where our results are of significant importance is statistics. Here, $\la$ plays the role of sample size, and high-dimensional Laplace-type integrals arise ubiquitously as normalizing constants, marginal likelihoods, and posterior expectations. As mentioned above, a central task is to approximate the posterior density $\pi \propto e^{-\la f}$: one wants to sample from it, compute expectations of observables against it, and evaluate its normalizing constant for model comparison. \emph{Our results address all three.} We construct an explicit approximation $\hat{\pi}_L$ to $\pi$ from which one can easily sample, we provide closed-form approximations to posterior expectations of smooth observables that avoid Monte Carlo error entirely, and we give rigorous asymptotic expansions of the normalizing constant that generalize the Bayesian Information Criterion (BIC) to higher order. Laplace-based surrogates such as the BIC are popular precisely because they provide explicit analytic approximations rather than black-box numerical estimates. Yet until now, theoretical backing for these approximations was limited to the $d^2/\la \to 0$ regime. Our results push these guarantees into 
substantially higher-dimensional regimes, arbitrarily close to the concentration threshold $d/\la\to0$. See Section~\ref{ssec:statistics_motivation} for more details and references.


Besides physics and statistics, high-dimensional Laplace-type integrals arise throughout science and engineering, including in molecular simulation and theoretical chemistry (e.g., partition functions and free-energy calculations) \cite{tuckerman2000,markland2018}, Bayesian inverse problems \cite{helin2022non}, and others.

Finally, our result has interesting connections to the theory of cumulants and to normalizing flows in machine learning.
The connection with the former is that we have solved the problem of bounding the remainder in cumulant expansions. Obtaining a high-dimensional remainder bound using cumulant theory (and the closely related theory of Gaussian chaos~\cite{janson1997gaussian, peccati2011, nourdin2012normal}) directly is deeply nontrivial and has not been done before. We have avoided this problem by finding an alternative route. See Section~\ref{sec:cumulant-intro} for more details.

Regarding normalizing flows, these are sequences of maps $T_1, \dots, T_L:\br^d\to\br^d$ with the property that if $Z\sim\mathcal N(0, I_d)$ then $T_1\circ\dots\circ T_L(Z)$ is approximately distributed according to a target distribution $\pi$~\cite{kobyzev2021nfreview,rezende2015normalizingflows}. Typically, $T_1,\dots,T_L$ are constructed using deep neural networks, the parameters of which are found by minimizing a loss~\cite{dinh2017realnvp}. Here, we approximate $\pi$ in a similar fashion, pushing forward a Gaussian distribution through a series of transformations. But the key differences with normalizing flows are that 1) our construction is explicit, not through minimizing some loss, and 2) we prove a rigorous error bound. Essentially the only structural assumption we need to obtain this result is that the target $\pi$ is a concentrating measure. 

\subsection{Application in physics}
\label{ssec:physics_motivation}

In a wide range of models in statistical physics and Euclidean QFT the \emph{partition function}, which is defined as
\be\label{eq:Z_field}
Z=\int \exp\bigl(-\beta\,\mathcal{H}(\phi)\bigr)\,\mathcal{D}\phi,
\ee
plays a central role. Here $\bt$ is inversely proportional to the temperature, $\CH$ is the Hamiltonian of the model, and $\int\dots\mathcal{D}\phi$ denotes functional (infinite-dimensional) integration over all allowed configurations $\phi$ of the model \cite[Sections 2.1]{brezin2010}, \cite[Sections 2.1]{Kardar2013}, \cite[Introduction to Chapter 2]{zinn-justin2021}. The normalized logarithm of the partition function is known as free energy: $\CF=-\beta^{-1}\log Z$. 
The functional integral in~\eqref{eq:Z_field} is typically interpreted as the limit of the appropriate finite dimensional integrals \cite[Sections 2.3]{Kardar2013}. After discretization one obtains an ordinary (finite-dimensional) integral over $d$ degrees of freedom,
\be\label{eq:Z_Laplace}
Z_{d}(\la)=
\int_{\mathbb{R}^{d}} \exp\!\bigl(-\la f(x)\bigr)\,dx,
\quad \mathcal{F}_{d}(\la)\;=\; -\la^{-1}\log Z_{d}(\la),
\ee
where $\la$ is a large prefactor coming from the problem \cite[Sections 2.3]{Kardar2013}. If the model is discrete to begin with (i.e., it has finitely many degrees of freedom $d$), one gets \eqref{eq:Z_Laplace} right away \cite[Sections 2.1 and 2.2]{Kardar2013} and \cite[Sections 1.1--1.3]{brezin2010}.


A basic approximation strategy, known as mean-field approximation, is to replace the integral in \eqref{eq:Z_Laplace} by its leading Laplace contribution near a
global minimizer $x_\star$ of $f$:
\be\label{eq:mf}
\log Z_{d}(\la) \approx
-\la f(x_\star)-\tfrac12\log\det\bigl(\nabla^2 f(x_\star)\bigr)-\tfrac d2 \log (\la/(2\pi)).
\ee
In physics texts, this step is typically justified by a combination of (i) the presence of a large parameter
(e.g.\ volume), and (ii) an \emph{a posteriori} self-consistency check that the discarded terms are “small”
in the regime of interest; see again \cite[Sec.~2.3]{Kardar2013} for an example of this approach.

To go beyond~\eqref{eq:mf}, one expands $f$ around $x_\star$,
\be\label{eq:Taylor}
f(x_\star + u) = f(x_\star)+\tfrac12\,u^\top H u +\sum_{k\ge 3}\tfrac1{k!}\,T_k[u^{\otimes k}],
\quad H=\nabla^2 f(x_\star),
\ee
rescales $u=\la^{-1/2}y$, and rewrites~\eqref{eq:Z_Laplace} as a Gaussian expectation:
\bs\label{eq:Gaussian_rewrite}
Z_{d}(\la)&=e^{-\la f(x_\star)}\big(\tfrac{\la}{2\pi}\big)^{-d/2}(\det H)^{-1/2}\;
\mathbb{E}\bigl[\exp\bigl(-V_\la(Y)\bigr)\bigr],\\
V_\la(Y)&:= \sum_{k\ge 3}\frac{\la^{1-k/2}}{k!}\,T_k[Y^{\otimes k}],
\quad Y\sim \mathcal{N}(0,H^{-1}).
\es
Taking logs yields a \emph{formal cumulant expansion} \cite[Section 3.10]{mccullagh2018}, \cite[Sections 3.1--3.6]{etingof2024}:
\be\label{eq:cumulant}
\log \mathbb{E}\bigl[e^{-V_\la(Y)}\bigr]
=
\sum_{m\ge 1}\frac{(-1)^m}{m!}\,\cum\!\bigl(V_\la(Y)^{\otimes m}\bigr),
\ee
where $\cum$ denotes the corresponding cumulant. In field-theoretic language, the terms in~\eqref{eq:cumulant} are precisely the \emph{loop corrections} to mean field: Wick expansion of the Gaussian moments \cite[Sections 5.1]{Kardar2013}, \cite[Section 1.1]{zinn-justin2021} produces a sum over Feynman diagrams, and the cumulant picks out the connected diagrams, which correct the free energy $\mathcal{F}_{d}(\la)$ order-by-order in $\la^{-1}$ \cite[Sections 9.1]{brezin2010}, \cite[Section 1.2 and Chapter 7]{zinn-justin2021}, \cite[Sections 3.13]{etingof2024}.

Physics literature makes extensive use of truncations of~\eqref{eq:cumulant} (or equivalent loop expansions). However, derivations are usually done without rigorous error control \cite[Sections 2.1 and 2.2]{Kardar2013} and \cite[Sections 1.1--1.3]{brezin2010}. For finite $d$, the error control follows from classical asymptotic analysis \cite{bh, wong}. For growing $d$, such expansions have never been justified except for some specific cases.

In  summary, {\it mean field approximation and loop/cumulant expansions are central computational tools in statistical physics, but rigorous remainder estimates (especially in large systems with many degrees of freedom, $d\gg1$) are often absent from the standard presentations.}

Our results provide a \emph{rigorous} version of the loop-correction program for~\eqref{eq:Z_Laplace} in a joint limit where both the dimension $d$ (number of effective degrees of freedom retained in the reduced description) and the Laplace prefactor $\la$ grow.  Concretely, under mild confining assumptions ensuring integrability and assuming a unique nondegenerate minimum $x_\star$ together with bounded operator norms of derivatives near $x_\star$, we obtain for every integer $L\ge 1$ an expansion of $\log Z_{d}(\la)$ through $L-1$ loops whose remainder is explicitly controlled (see \eqref{main res intro}) 
\bs\label{our_remainder}
\log Z_{d}(\la) =&
\Bigl[\text{mean field}\Bigr] +
\sum_{\ell=1}^{L-1}\Bigl[\text{$\ell$-loop / $\ell$th cumulant correction}\Bigr]
+\mathcal{R}_{L}(d,\la),\\
|\mathcal{R}_{L}(d,\la)|
\lesssim & d^{L+1}/\la^{L},
\es
uniformly over the stated class of $f$.  The scaling condition $d^{L+1}/\la^{L}\to 0$,
$d,\la\to\infty$ is exactly the statement that \emph{$(L-1)$-loop-corrected mean field has a provable accuracy guarantee} in a growing-system regime. From a physics viewpoint,~\eqref{our_remainder} can be read as supplying the missing ``error bars'' for a procedure that is otherwise typically justified heuristically (or by numerics), thereby turning the loop expansion into a rigorously justified approximation scheme in regimes where the effective number of degrees of freedom grows with the large parameter in the exponential.

\subsection{Application in statistics}
For the sake of brevity, we focus on Laplace-type integrals and densities in the context of \emph{Bayesian} statistics; however, these quantities frequently arise in frequentist statistics as well.
\label{ssec:statistics_motivation}
\subsubsection{Statistics context and state of the art}
Laplace-type integrals $\int_{\br^d} g(x)e^{-\la f(x)}\dd x$ and probability distributions $\pi(x)\propto e^{-\la f(x)}$ are omnipresent in statistics. Here, $\la$ (more commonly denoted $n$ in statistics) plays the role of sample size, and $f$ depends weakly on $\la$~\cite{tierney_kadane_1986}. In Bayesian inference, $\pi(x)=P(x\mid \mathrm{data})$ is the posterior probability distribution of an unknown parameter $x$ given $\mathrm{data}$, consisting of $\la=n$ independent data points. Once the posterior has been specified, one is typically interested in (1) computing summary statistics, which take the general form $\E_{X\sim\pi}[g(X)]$, (2) sampling from $\pi$, and (3) computing the normalizing constant $\int e^{-\la f(x)}\dd x$ itself~\cite{bda3}. Summary statistics distill information in the posterior; samples enable exploration and uncertainty quantification. The normalizing constant, known as the model evidence~\cite{kass_raftery_1995}, underpins Bayesian model selection. Here, one finds the best model by optimizing the evidence over candidate models~\cite{kass_raftery_1995, wasserman2000bayesian},~\cite[Chapter 5]{ando2010bayesian}. All three tasks become particularly demanding in the growing-$d$ regime that now arises routinely in applications.

The problem of how to do these computations efficiently has been actively studied for decades. A core approach is to replace $\pi$ with a tractable approximation $\hat\pi$, and the literature spans both theoretical and numerical methods; see, e.g.,~\cite{tierney_kadane_1986,huggins2018practical,schillings2020convergence,durante2024skewed, blei2017variational} among many others. Within the theoretical analyses, a recent line of work focuses on the dependence of the total variation error $\mathrm{TV}(\pi,\hat\pi)$ on both $d$ and $\la$, either by sharpening error control for existing approximations $\hat\pi$ or by constructing improved approximations. In particular,~\cite{spok23,dehaene2019deterministic,helin2022non} showed that for the standard Laplace approximation $\hat\pi=\mathcal N\!\big(x_\star,(\la\nabla^2 f(x_\star))^{-1}\big)$, the TV error scales as $d\sqrt d/\sqrt\la$. Via a tighter analysis,~\cite{kasprzak2025, katskew,katsevich2025unified} improved this dimension dependence to $d/\sqrt\la$, and~\cite{KRVI} showed that the closely related Gaussian variational-inference approximation also achieves TV error $d/\sqrt\la$. Meanwhile,~\cite{katskew} and~\cite{durante2024skewed} proposed new approximations incorporating third-order derivative information, improving the $\la$ dependence to $\mathrm{TV}(\pi,\hat\pi)\les d^2/\la$ and $\les d^3/\la$, respectively.

While approximating $\pi$ by a tractable $\hat\pi$ is natural for sampling and for evaluating $\E_{X\sim\pi}[g(X)]$ when $g$ is nonsmooth, one can often do better when $g$ is smooth by working directly with the LE. Indeed, $\E_{X\sim\pi}[g(X)]$ is a ratio of two Laplace-type integrals. Expanding both the numerator and the denominator and then taking the ratio yields an explicit, fully deterministic approximation to $\E_{X\sim\pi}[g(X)]$, avoiding the Monte Carlo error inherent in estimating expectations via samples from $\hat\pi$. Despite this advantage, LEs have received comparatively little attention as a tool for computing posterior expectations. Notable exceptions are the fixed-$d$ analyses in~\cite{tierney_kadane_1986, tierney_kass_kadane_1989}.

In contrast, LE-based approximations are widely used for model selection problems that maximize the normalizing constant (model evidence) $\int e^{-\la f(x)}\dd x$ over a tuning parameter. A key attraction is that the LE provides an explicit analytic surrogate objective as a function of the tuning parameter, enabling efficient optimization. This stands in sharp contrast to methods that typically provide only black-box numerical access to the objective, such as Markov Chain Monte Carlo~\cite{robert1999monte}. However, this convenience of the LE is meaningful only when the approximation error is controlled in a dimension-dependent manner. In particular, concerns about the reliability of the Bayesian Information Criterion (BIC) in high dimensions have been noted explicitly. The BIC introduces an additional approximation on top of the LE to further simplify the optimization objective. In~\cite[p. 377]{drton2017model}, the authors write ``for models with a large number of predictors [large $d$] it is no longer clear that some version of the BIC, or perhaps rather a Laplace approximation, accurately approximates a marginal likelihood''.

While recent work has begun to provide rigorous, dimension-dependent guarantees for high-dimensional LEs for normalizing constants~\cite{barber2016laplace,tang2025,katsevich2025a}, these require more restrictive scalings (at best, that $d^2/\la\ll 1$).

\subsubsection{Implication of our results in statistics}
\emph{Our work essentially completes the program of finding tractable analytic approximations of posterior densities, expectations, and normalizing constants in the high-dimensional but concentrating regime.} We propose a combined approach for these related problems which allows dimension to grow relative to $\la$ arbitrarily close to the concentration threshold.

For sampling, we construct a sequence of arbitrarily accurate approximations $\hat\pi_L$ which can be sampled from using an explicit algorithm. For computing expectations of smooth functions, we give an arbitrarily accurate closed-form formula based on the LE of numerator and denominator. For normalizing constants,~\eqref{main res intro} directly applies with $g\equiv1$. The approximation error for all these tasks is $d^{L+1}/\la^L$, which has the crucial feature that increasing $L$ not only improves the accuracy but also expands the range of applicable $d$, since we can take $d$ as large as $o(\la^{L/L+1})$. 

We now give more detail about our methods. 

\paragraph*{Sampling and nonsmooth expectations.}

To approximately sample from $\pi$, the algorithm consists of drawing $Z_i$ from $\mathcal N(0,\la^{-1}I_d)$ and mapping it through one of the $x_L$, $L=1,2,3,\dots$, depending on the desired accuracy. To approximate expectations of nonsmooth functions $g$, the samples $X_i=x_L(Z_i) \sim \hat{\pi}_L$ can also be used, via 
\begin{equation}\label{eq:monte-carlo-approx}
  \E_{X \sim \pi}[g(X)]
    \approx \E_{X \sim \hat{\pi}_L}[g(X)]
    \approx \frac{1}{N}\sum_{i=1}^{N} g(X_i).
\end{equation}

\paragraph*{Normalizing constants and smooth expectations.}
Approximating normalizing constants is immediate using~\eqref{main res intro} with $g\equiv1$, so we don't discuss it further. To approximate expectations $\E_{X \sim \pi}[g(X)] = \int ge^{-\la f}/\int e^{-\la f}$ of smooth $g$, we use~\eqref{main res intro} for the numerator and denominator. This gives the approximation
$\E_{X \sim \pi}[g(X)]\approx\exp\bigl(\sum_{k=1}^{L-1}[b_k(f,g) - b_k(f,1)]/\lambda^k\bigr)$, which has accuracy $\CO(d^{L+1}/\la^L)$ uniformly over sufficiently smooth $g$ with bounded derivatives near $x_\star$. 

This estimate improves on~\eqref{eq:monte-carlo-approx} in two ways. First, it does not incur any sampling error. Second, it is constructed using fewer derivatives of $f$:
$2L-1$ versus $2L+1$ for $\hat{\pi}_L$. Intuitively, $\hat{\pi}_L$ does not exploit smoothness of $g$ and must compensate with more information from $f$. In practice, this matters because $f$ encodes the data likelihood and is expensive to differentiate, while the observable $g$ is typically simple. For example, when $L = 2$, the closed-form estimate $\exp([b_1(f,g) - b_1(f,1)]/\la)$ has accuracy $d^3/\lambda^2$ using only three derivatives of $f$, whereas $\hat{\pi}_1$ also uses three derivatives but achieves only the lower accuracy $d^2/\lambda$. Thus using the closed-form asymptotic estimate \emph{improves the accuracy by a factor of $d/\la$ over sampling approaches without requiring any more derivatives}. We emphasize that most approaches in the literature are of the sampling type and achieve accuracy $d^2/\la$ (our $\hat\pi_1$ and~\cite{katskew}) or $d^3/\la$~\cite{durante2024skewed} at best, even disregarding the Monte Carlo error. See Section~\ref{subsec:meas:discuss} for a detailed comparison to the two works closest to ours,~\cite{durante2024skewed, katskew}.

\subsection{Connection to cumulant theory}
\label{sec:cumulant-intro}
As stated above, the coefficients $b_k$ in~\eqref{main res intro} can be derived from formal cumulant expansions. Although computing the \emph{coefficients} using cumulants is easy, bounding the \emph{remainder} in the expansion of $\log I(\la)$ using cumulant theory seems deeply nontrivial. Indeed, cumulants have been thoroughly studied in the statistics, physics and combinatorics literature, yet rigorous remainder bounds on cumulant expansions in the growing $d$ regime are virtually non-existent. This suggests that bounding the remainder through the lens of cumulants may be intractable. In contrast, our change-of-variables approach is not only tractable, but also avoids any heavy machinery, as mentioned above.

Our remainder bound makes rigorous the well-known observation that, based on the terms, expanding the cumulant generating function (cgf) is advantageous to expanding the moment generating function (mgf). Specifically, it is well-known that each cumulant is given by a sum of a fewer number of summands than the corresponding moment. For example at the end of Section 3.10.1 of~\cite{mccullagh2018}, McCullagh writes that the formula for cumulants is a sum only over ``connected pairs" of bi-partitions, as opposed to the formula for moments. In~\cite[Chapter 5.3]{Kardar2013}, the author writes, ``When calculating cumulants, only fully connected diagrams (without disjoint pieces) need to be included. This is a tremendous simplification." 

The simplification  due to this summing over fewer terms in fact leads to a substantial improvement. By analyzing the cumulants and moments directly in our setting, it is possible to show that the moments that contribute to the $k$th coefficient in the $I(\la)$ expansion are generically of order $\CO(d^{2k})$, while the cumulants that contribute to the $k$th coefficient in the $\log I(\la)$ expansion are generically of order $\CO(d^{k+1})$. Of course, this analysis of the terms does not rigorously prove anything about the remainders. However, it naturally leads to the hypothesis that the remainders of the two expansions behave like the first terms dropped from the truncation, i.e. that the $L$th order remainders scale as $\CO(d^{2L}/\la^L)$ and $\CO(d^{L+1}/\la^L)$, respectively. This is precisely what the previous work~\cite{katsevich2025a} and the present paper rigorously prove.

Thus, as mentioned above, expanding the logarithm of $I(\la)$ is what allows us to relax the $d^2\ll\la$ requirement in~\cite{katsevich2025a}. Given our result~\eqref{main res intro}, it is very easy to see why $d^2\ll\la$ is necessary for an \emph{additive} expansion of $I(\la)$ to any order. Indeed, we have shown in~\eqref{main res intro} that, to leading order, $I(\la)=\exp(\CO(d^2/\la))$. Thus if we Taylor expand the exponential to obtain an additive expansion of $I(\la)$, powers of $d^2/\la$ are present throughout. There is no finite number of terms we can \emph{subtract} from $I(\la)$ to remove the dependence on $d^2/\la$. However, we \emph{can} remove the dependence on $d^2/\la$ simply by \emph{dividing} $I(\la)$ by the single term $\exp(b_1(f,g)/\la)$. The multiplicative remainder is then of order $\exp(\CO(d^3/\la^2))$.

In Section~\ref{sec:examples}, we illustrate the power of expanding $\log I(\la)$ in two concrete examples:  a quartic perturbation of a Gaussian exponent, and a logistic-regression-type likelihood motivated by statistics. 

\subsection{Related work}
There is a related body of work on Laplace-type integrals arising in statistics (e.g. posterior normalizing constants and marginals) and statistical physics  (partition functions) in the ``proportional asymptotics" regime. Here, $d/\la$ converges to a constant. The regime lies beyond the concentration threshold, so these integrals can only formally be considered to be Laplace-type. Due to the lack of concentration, entirely different techniques are required to study such asymptotics. See~\cite{lee2025clt,saenz2025characterizing,celentano2023mean, mukherjee2024naive,mukherjee2022variational}, as well as the referenced cited therein.


\subsection*{Organization}
The paper is organized as follows. Section~\ref{sec:notation} introduces the notation and conventions used throughout. In Section~\ref{sec:main res}, we describe the problem setting and state the main result, and Section~\ref{sec:proof} outlines the proof strategy. Section~\ref{sec:in} derives the expansion of the main contribution to the integral coming from a neighborhood of zero, while Section~\ref{sec:cum} computes the expansion coefficients in terms of cumulants. The tail contribution is estimated in Section~\ref{sec:out}. Section~\ref{sec:meas} presents our results 
on the approximation of $\pi\propto e^{-\la f}$ and on computing expectations against $\pi$. Section~\ref{sec:examples} compares our results with related work in the literature, and presents two applications of our expansion. Various technical results are collected in the appendices.

\section{Notation and conventions}\label{sec:notation} 
When we write e.g. $\sum_{k\geq2}a_k$, the sum is over a finite number of $k$'s. Also, any sum of the form $\sum_{k=1}^0 a_k$ is understood to be omitted.

\begin{definition}\label{def:Akj}
A \emph{tensor} $A_{k\to j}$ is a multilinear map taking in $k$ vectors in $\br^d$ and returning a scalar if $j=0$, a vector in $\br^d$ if $j=1$, or a $d\times d$ matrix if $j=2$. We will typically use the letters $G$ and $F$ instead of $A$. We say $A_{k\to j}$ is \emph{symmetric} if for all permutations $\sigma:\{1,\dots,k\}\to\{1,\dots,k\}$ and vectors $x_1,\dots,x_k\in\br^d$ it holds $A_{k\to j}[x_1,\dots,x_k]=A_{k\to j}[x_{\sigma(1)},\dots,x_{\sigma(k)}]$. Note that symmetry refers to the input space (permuting the $k$ arguments) rather than to the output space. In particular, $A_{k\to 2}$ can be symmetric even if $A_{k\to 2}[x_1,\dots,x_k]$ is not a symmetric matrix. 

When $j=0$, we often omit $\to 0$ in the subscript and simply write $A_k$. When $x_1=\dots=x_k=x\in\br^d$, we write $A_{k\to j}[x^{\otimes k}]$ instead of $A_{k\to j}[x,x,\dots,x]$. We also allow $k=0$; then $A_{0\to 1}$ is a constant vector-valued function and $A_{0\to2}$ is a constant matrix-valued function. 
\end{definition}

In some contexts, the term \emph{tensor} is used to refer to multilinear forms (the case $j=0$). Here we adopt the convention of Definition~\ref{def:Akj} and also call $A_{k\to 1}$ and $A_{k\to 2}$ tensors, i.e., vector- and matrix-valued multilinear maps.

The operator norm of a symmetric $k$-th order tensor $A_k$, i.e., $A_{k\to0}$, is given by \cite{zhang2012}
\be\label{T norm}
\Vert A_k\Vert_{\op}:=\sup_{\|u\|=1} A_k[u^{\otimes k}].
\ee
For a function $f\in C^k(\br^d)$, the $k$-th derivative tensor is
\be
(\nabla^kf(x))_{i_1\dots i_k} = \pa_{x_{i_1}}\dots\pa_{x_{i_k}}f(x). 
\ee

\noindent We let $\e=d/\la$, and fix an integer $L\ge 1$.  


\section{Statement of result}\label{sec:main res}

We study the asymptotics of the integral
\be\label{main int v0}
I(\la)=\left(\frac{\la}{2\pi}\right)^{d/2}\int_{\br^d}g(x)e^{-\la f(x)}\dd x
\ee
as $d,\la\to\infty$. We make the following assumption on $f$ and $g$.

\begin{assumption}\label{ass:fg}$\hspace{1cm}$
\begin{enumerate}[label=(\arabic*)]
\item\label{Ial} There exists $r_0>0$ independent of $d$ and $\la$ such that 
$f$ satisfies 
\bs\label{eq:Taylor-f}
f(x)=&f_{2L+1}(x)+\mathcal{R}_{2L+2}(x),\ \norm{x}\le r_0,\\
f_{2L+1}(x)=& \tfrac12 \norm{x}^2 + \textstyle\sum_{k=3}^{2L+1}\tfrac{1}{k!}\nabla^k f(0)[x^{\otimes k}].
\es
The operator norms and the remainder are bounded:
\bs\label{Fk norms}
\|\nabla^kf(0)\|_{\op} \le& \argc{f}, \ 3\le k\le 2L+1;\quad 
|\mathcal{R}_{2L+2}(x)| \le \argc{f} \norm{x}^{2(L+1)},\ \norm{x}\le r_0.
\es 
\item\label{Ial g}We have $g(0)=1$. For the same $r_0$ as in~\ref{Ial}, we have 
\bs\label{eq:Taylor-g}
\log g(x)=& \textstyle\sum_{k=1}^{2L-1} \tfrac{1}{k!}\nabla^k(\log g)(0)[x^{\otimes k}] + \mathcal{R}_{2L}^{(g)}(x),\ \norm{x}\le r_0.
\es
The operator norms and the remainder are bounded
\bs\label{Gk norms}
\|\nabla^k\log g(0)\|_{\op} \le& \argc{g}, \ 1\le k\le 2L-1;\quad 
|\mathcal{R}_{2L}^{(g)}(x)| \le \argc{g} \norm{x}^{2L},\ \norm{x}\le r_0.
\es
\item There exists a constant $\kappa>0$ independent of $d,\la$ such that
\be\label{eq:coercive gf}
|g(x)|e^{-\la f(x)}\leq \exp\left(\argc{g}-\la\kappa\min\left\{\|x\|^2/2, \sqrt{\e}\norm{x}, d^{-1/(2L)}\log(1+\norm{x})\right\}\right),\quad\forall\,x\in\br^d.
\ee 
Without loss of generality, assume $\kappa\leq 1/e$, and that $\argc{g}$ is the same as in part~\ref{Ial g}.
\end{enumerate}
\end{assumption}
In~\eqref{eq:coercive gf} and everywhere below, recall that $\e=d/\la$. 

The case of a general minimizer $x_\star$ of $f$ and general positive definite $\nabla^2f(x_\star)$ can be reduced to the one assumed above by an affine change of variables. By linearity, the normalization condition $g(0)=1$ is not restrictive. A sufficient condition for~\eqref{eq:coercive gf} to be satisfied is
\begin{align}
f(x)&\geq\kappa_f\min\left\{\|x\|^2/2, \sqrt{\e}\norm{x}, d^{-1/(2L)}\log(1+\norm{x})\right\},\quad x\in\br^d,\label{f-suff}\\
|g(x)|&\leq \min\left\{e^{\kappa_g\la\|x\|^2/2}, e^{\kappa_g\sqrt{d\la}\|x\|}, (1+\|x\|)^{\la\kappa_g/d^{1/(2L)}}\right\},\quad \|x\|\geq r_0,\label{g-suff}\\
\kappa&:=\kappa_f-\kappa_g>0.
\end{align}
When $x$ is small, the minimum in the first line is equal to $\min\{\|x\|^2/2, \sqrt{\e}\norm{x}\}$. When $x$ is large, the minimum is given by $d^{-1/(2L)}\log(1+\norm{x})$. Thus we only require that $f$ grows logarithmically at infinity. Meanwhile, $g$ can grow at most polynomially at infinity under~\eqref{g-suff}, but the power of the polynomial can become arbitrarily large as $\la\to\infty$. \\

For positive scalars $a,b$, the notation $a\lesssim b$ and $a=\CO(b)$ both mean that $a\le cb$, where the suppressed constant $c$ can depend on the constants $\argc{f}$ and $\argc{g}$ appearing in Assumption~\ref{ass:fg} but not on $d,\la,r_0,\kappa$. 
Since $L\ge 1$ is fixed, the dependence of constants on $L$ is ignored. In a similar vein, when we say ``for all $\delta$ sufficiently small, it holds....'' we mean there exists $c(\argc{f},\argc{g})$ such that if $0<\delta\leq c(\argc{f},\argc{g})$, then the statement after the ellipses is true. Here, $\de$ denotes an arbitrary small parameter that is relevant to the discussion at hand. In most cases, $\de=R\sqrt\e$ (see below for the definition of $R$).

The following is our first main result.
\begin{theorem}\label{thm:main res} Fix any $L\ge 1$. Suppose $f,g$ satisfy Assumption~\ref{ass:fg} and $d\geq 2L$. There exists a sufficiently large $c^-=c^-(\argc{f},\argc{g})>0$ and sufficiently small $c^{+}=c^+(\argc{f},\argc{g})>0$ such that for any $R,d,\la$ satisfying 
\be\label{R ineq v1}
R\geq c^-\max\left(1,\tfrac{1}{\kappa}\left[\sqrt{\log\tfrac1\kappa}+\log R+\tfrac{\log\la}d\right]\right),\quad \frac{(R^2d)^{L+1}}{\la^L}\leq c^+,\quad R\sqrt{d/\la}\leq r_0,\ee
it holds $I(\la)>0$, and
\bs\label{main res}
\Big|\log I(\la) & -\sum_{k=1}^{L-1} b_k(f,g)\la^{-k}\Big|
\les \frac{(R^2d)^{L+1}}{\la^L}.
\es
The coefficients $b_k(f,g)$ depend only on $\nabla^\ell f(0)$, $\ell=3,\dots,2k+2$, $\nabla^\ell \log g(0)$, $\ell=1,\dots,2k$, but not explicitly on $d$ or $\la$, and their formula is given in Lemma~\ref{prop:ad}. Moreover, $|b_k(f,g)|\les d^{k+1}$.
\end{theorem}
To illustrate the convention introduced before Theorem~\ref{thm:main res}, we note that the suppressed constant in \eqref{main res} and in the bound on $|b_k(f,g)|$ may depend on $\argc{f}$, $\argc{g}$ only, but not on $d,\la,r_0,\kappa$. Note also that, according to our convention, the sum with respect to $k$ in \eqref{main res} is omitted if $L=1$.

We give a closed-form formula for $b_1(f,g)$ in Section~\ref{sec:c1}. As is seen, in the worst case the first inequality in \eqref{R ineq v1} forces $R\sim\log\la$, up to a factor depending on $\argc{f},\argc{g},\kappa$. This occurs when $d$ remains bounded as $\la\to\infty$. On the other hand, if $d\geq c(\argc{f},\argc{g},\kappa)\log\la$, then \eqref{R ineq v1} is satisfied with $R$ being a constant independent of $d$ and $\la$.

\section{Proof outline}\label{sec:proof}

\subsection{Key steps}
Here, we outline the main ideas of our approach. Without loss of generality, we suppose $x_\star=0$ is the global minimizer of $f$, with $\nabla^2f(0)=I_d$, so that $f(x)=\|x\|^2/2+o(\|x\|^2)$ near zero. For simplicity, we only describe the proof in the case $g\equiv1$. Incorporating a non-constant $g$ presents no real challenges. On the surface, our proof begins similarly to classical fixed-$d$ proofs of the expansion based on the Morse Lemma~\cite{wong}. This lemma states that there is a change of variables $x=x(t)$ making the exponent $f(x(t))$ an exact quadratic. This approach, in its pure form, seems to be intractable because the coordinate transformation from the Morse Lemma is difficult to work with. \\

\noindent Instead, we use a variation of the approach. Fix any $L\ge1$. \\
\textbf{Step 1: initial change of variables.} We begin by constructing an explicit local polynomial change of variables $X(t)$, of the form $X(t)=t+\CO(\|t\|^2)$, $t\to0$, which makes the exponent ``more quadratic'' but not exactly quadratic. Specifically, it eliminates the third through $(2L+1)$st order terms in the Taylor expansion of $f$ around the minimizer, so that $-\la f(X(t))=-\la\norm{t}^2/2+\CO(\la\norm{t}^{2L+2})$. But $\CO(\la\norm{t}^{2L+2})$ is of order $O(d^{L+1}/\la^L)$ (and therefore negligible) in the region $\|t\|\les \sqrt{d/\la}$ where the integral concentrates. Thus $\CO(\la\norm{t}^{2L+2})$ can be discarded.

The price we pay for this nice change of variables is the appearance of the Jacobian $\detm(X^\prime(t))$ of the coordinate change, which we bring into the exponent. Thus the new exponential function is, upon throwing out the negligible $\|t\|^{2L+2}$ term, given by $\exp(-\la\|t\|^2/2 + \log\detm(X^\prime(t)))$. Crucially, however, $\log\detm(X^\prime(t))$ scales only as $d\ll \la$. Thus the log-Jacobian does not significantly affect the quadratic exponent created by the change of variables. Specifically, we may write $\log\detm(X^\prime(t))=dh(t)$, where $h(t)$ can be Taylor-expanded as $h(t)=\sum_{k\geq1}F_k[t^{\otimes k}]+(\text{negligible remainder})$ for some tensors (multi-linear forms) $F_k$ with bounded operator norms. Here and below, sums in which the upper limit has not been explicitly indicated are understood to mean finite sums. The above arguments lead to
\bs\label{step1-change intro}
I(\la)= & e^{\CO(d^{L+1}/\la^L)}\Big(\frac{\la}{2\pi}\Big)^{d/2}\int_{\us_1} \exp\big(-\la E_1(t)\big)\dd t,\\
E_1(t):=&\frac12\|t\|^2+\e\sum_{k\geq1}F_k[t^{\otimes k}],
\es where $\e=d/\la$, $\us_1=\{\|t\|\les \sqrt{\e}\}$, and we have discarded the negligible integral over $\us_1^c$. Now, we must somehow deal with the terms $\e F_k[t^{\otimes k}]$, $k\geq1$. The linear and quadratic terms ($k=1,2$) are not a problem, since they can be combined with $\|t\|^2/2$ to give a new Gaussian measure, and integrating against Gaussians is tractable. For $k\ge 2L$, the terms $\la\e|F_k[t^{\otimes k}]|$ are sufficiently small uniformly over $t\in\us_1$, so that $\exp(-\la\e F_k[t^{\otimes k}])$ can be discarded from the integral (meaning, absorbed in the multiplicative remainder $e^{\CO(d^{L+1}/\la^L)}$). But there is an intermediate range of $k$'s, namely, $3\le k\le 2L-1$, for which additional massaging is needed. We do this in the next step. In what follows, all our formulas ignore terms that lead to multiplicative factors $e^{\CO(d^{L+1}/\la^L)}$ in $I(\la)$. \\

\noindent
\textbf{Step 2: iterative refinement.} We construct a polynomial change of variables of the form $T_1(s)=s+\e\varphi_1(s)$, with $\varphi_1(s)=\CO(\|s\|^2)$, analogous to $X(t)$ above. By choosing $\varphi_1$ appropriately, we can ensure that this change of variables increases the power of $\e$ in front of $F_k[t^{\otimes k}]$ for as many $k\geq3$ as we like. Thus
\be\label{E1intro}
E_1(T_1(s))=\e s+\e s^2+\tfrac12\|s\|^2+\e^2\sum_{k= 3}^{2L-3}s^k.
\ee where $s^k$ is a mnemonic form for $F_k[s^{\otimes k}]$, and $F_k$ is some tensor with bounded operator norm. Furthermore, due to the form of $T_1(s)$, we have 
\be\label{1laintro}
\tfrac1\la\log\det(T_1^\prime(s))=\e^2\sum_{k=2}^{2L-3}s^k,
\ee 
informally. Here, one of the powers of $\e$ comes from $\e\varphi_1(s)$, and the other power of $\e$ comes from $\frac1\la\times d$, since $\log\det\sim d$. Since the powers of $\e$ in~\eqref{1laintro} are greater than or equal to the corresponding powers of $\e$ in~\eqref{E1intro}, adding $\frac1\la\log\det(T_1^\prime(s))$ to $E_1(t_1(s))$ does not change the structure of the latter.

Therefore, implementing the change of variables, the above argument gives
\bs\label{step2-change intro}
I(\la)= & e^{\CO(d^{L+1}/\la^L)}\Big(\frac{\la}{2\pi}\Big)^{d/2}\int_{\us_2} \exp\big(-\la E_2(t)\big)\dd t,\\
E_2(t):=&\e t+\e t^2+\tfrac12\|t\|^2+\e^2\sum_{k=3}^{2L-3}t^k,\quad
\us_2:=T_1^{-1}(\us_1).
\es 
We can now repeat the procedure, each time increasing the power of $\e$ in front of $\sum_{k\geq3}t^k$. Finally, we arrive at $E_L(t)=\e t+\e t^2+\frac12\|t\|^2$ in the exponent. At this point, all terms in the sum with respect to $k$ have been absorbed in the multiplicative remainder $e^{\CO(d^{L+1}/\la^L)}$. Recalling that $\e t+\e t^2$ is a mnemonic for $\e a^\top t+ \e t^\top Bt$, where the norms of $a$ and $B$ are bounded, we conclude that
\bs\label{ELintro}
I(\la)=& e^{\CO(d^{L+1}/\la^L)}\Big(\frac{\la}{2\pi}\Big)^{d/2}\int_{\us_L} \exp\left(-\la\left[\e a^\top t+ \e t^\top Bt+\tfrac12\|t\|^2\right]\right),\\
\us_L:=&T_{L-1}^{-1}\circ \dots \circ T_1^{-1}(\us_1).
\es


\noindent
\textbf{Step 3: complete the square.} We incur a negligible error by extending the integral over $\us_L$ in~\eqref{ELintro} to $\br^d$. But the integral over $\br^d$ is now \emph{exactly computable}, since it becomes a Gaussian integral upon completing the square. We obtain precisely that
$$I_{\mathrm{in}}(\la)=e^{\CO(d^{L+1}/\la^L)}\exp\left(\tfrac12\la\e^2a^\top (I_d+2\e B)^{-1}a-\tfrac12\log\det (I_d+2\e B)\right).$$ Finally, the second exponential factor above gives the terms in~\eqref{main res intro}. To show this, we prove that $a$ can be written as $a=a_0+\e a_1+\e^2a_2+\dots$, and $B$ as $B=B_0+\e B_1+\e^2B_2+\dots$, where the $a_k$ and $B_k$ do not depend on $\e$, and have bounded norms. Substituting this into the exponent in parentheses leads to $\exp(\la\sum_{k\geq2}c_k\e^k +d\sum_{k\geq1}c_k'\e^k)$, where the first sum stems from $\tfrac12\la\e^2a^\top (I_d+2\e B)^{-1}a$ and the second sum stems from $-\tfrac12\log\det (I_d+2\e B)$. Combining the two sums gives the right orders of magnitude $d^2/\la+d^3/\la^2+\dots$ for the terms in~\eqref{main res intro}.

It is worth noting that our proof avoids the heavy Gaussian concentration machinery for Lipschitz functions (e.g., via log-Sobolev/Herbst-type arguments), in contrast to the strongest result to date \cite{katsevich2025a}.

\noindent
\textbf{Step 4: computing the coefficients.}
In Step 3 above, we have quantified the orders of magnitude of the expansion coefficients. Next, we show that these coefficients can be expressed in terms of cumulants. To do so, it suffices to consider $f,d$ fixed, view $\log I(\la)$ purely as a function of $t=\la^{-1/2}$, and take the derivatives of this function with respect to $t$. As a first step towards this goal, we write $I(\la)$ in the following form:
\be I(\la)=e^{\CO(t^{2L})}\E\left[\exp\left(tf_3(Z)+t^2f_4(Z)+\dots+t^{2L-1}f_{2L+1}(Z)\right)\ind\{\|Z\|\leq C\log(1/t)\}\right],\ee where $Z\sim\mathcal N(0, I_d)$ and $f_k(x)=-\frac{1}{k!}\nabla^kf(0)[x^{\otimes k}]$. Using this representation, we expand $\log I(\la)$ in powers of $t$. We show that the coefficients up to order $2L$ in this expansion coincide with the coefficients in the formal expansion of
$\log \E[\exp(\sum_{k=1}^{2L-1}t^kf_{k+2}(Z))]$ in powers of $t$. This latter expansion is formal since the expectation may not be finite. The reason the coefficients coincide is that the discrepancy between the coefficients of the two expansions is controlled by $P(\|Z||\geq C\log(1/t))\sim\exp(-\log^2(1/t))=o(t^M)$ for any $M$. Finally, we observe that the coefficients of the latter formal expansion can be expressed in terms of \emph{joint cumulants} of $f_k(Z)$, $k=3,\dots,2L+1$.

\subsection{Roadmap}\label{sec:roadmap}
In this section, we make the above proof outline precise, 
by breaking up the proof of Theorem~\ref{thm:main res} into three lemmas. These lemmas are stated here and proved in the subsequent sections, as indicated below the statement of each lemma.
\begin{definition}[Local and tail integrals]\label{def:loc}
Define 
the following local and tail integrals.
\begin{align}
I_{\mathrm{in}}(\la)&=\left(\frac{\la}{2\pi}\right)^{d/2}\int_{\|x\|\leq R\sqrt\e}g(x)e^{-\la f(x)}\dd x,\label{eq:in}\\
I_{\mathrm{out}}(\la)&=\left(\frac{\la}{2\pi}\right)^{d/2}\int_{\|x\|>R\sqrt\e}g(x)e^{-\la f(x)}\dd x.\label{eq:out}
\end{align} 
\end{definition}
We break up the proof of Theorem~\ref{thm:main res} into three key lemmas.
\begin{lemma}\label{lma:in}
Fix any $L\ge1$. Suppose parts~\ref{Ial} and~\ref{Ial g} of Assumption~\ref{ass:fg} hold. There exist large enough $c^-=c^-(\argc{f},\argc{g})>0$ and small enough $c^{+}=c^{+}(\argc{f},\argc{g})>0$ such that if $R$ in~\eqref{eq:in} satisfies 
\be\label{R ineq v2}
R\geq c^-\sqrt{1\vee\frac{\log\la}{d}},\qquad R\sqrt{\e}\leq\min(r_0, c^+),\ee then $I_{\mathrm{in}}(\la)>0$ and
\begin{align}\label{in}
\left|\log I_{\mathrm{in}}(\la)-\sum_{k=1}^{L-1}b_k(f,g,d)\la^{-k}\right| \les \frac{(R^2d)^{L+1}}{\la^L}.
\end{align}
The coefficients $b_k(f,g,d)$ do not explicitly depend on $\la$ and satisfy $|b_k(f,g,d)|\les d^{k+1}$. 
\end{lemma}
Lemma~\ref{lma:in} is proved in Section~\ref{sec:in}. Let $\alpha=(\alpha_1,\dots,\alpha_M)$ be a multiindex with $\alpha_j\geq 0$ for all $j$. Let $|\alpha|=\alpha_1+\dots+\alpha_M$, and $\pa_u^\alpha=\pa_{u_1}^{\alpha_1}\dots\pa_{u_M}^{\alpha_M}$.
\begin{lemma}\label{prop:ad}
The coefficients $b_k(f,g,d)$, $k=1,\dots, L-1$ from Lemma~\ref{lma:in} do not explicitly depend on $d$. They are given as follows. Let $M\leq 2L-2$ be even and $Z\sim\mathcal N(0, I_d)$. If $d\geq2L$, then
\be
b_{\frac M2}(f,g,d)=b_{\frac M2}(f,g)=\sum_{\substack{\alpha_1,\dots,\alpha_M\geq0\\\sum_{i=1}^Mi\alpha_i=M}}\frac{\cum(p_\alpha(Z))}{\prod_{i=1}^M\alpha_i!(i!)^{\alpha_i}}.\ee Here, $\cum(p_\alpha(Z))$ is defined at the beginning of Section~\ref{sec:cum}. Furthermore, $b_{M/2}(f,g)$ depends on derivatives of $f$ of order $3,\dots,M+2$ and derivatives of $\log g$ of order $1,\dots, M$.
\end{lemma}
Lemma~\ref{prop:ad} is proved in Section~\ref{sec:cum}. Although we have used somewhat different notation, the above formula can be shown to coincide with that given in~\cite[Chapter 3.10.2]{mccullagh2018}. 

\begin{remark}\label{rk:Zpower}
The above formula for $b_{M/2}(f,g)$ is the coefficient in front of $t^M$ in the formal power series expansion of $\log\E\left[\exp(h(t,Z))\right]$ in powers of $t$, where $h(t,Z)=\log g(tz) -\left(t^{-2}f(tz)-\|z\|^2/2\right)$ and $Z\sim\mathcal N(0, I_d)$.
\end{remark}

\begin{lemma}\label{lma:out}
Under the assumptions of Theorem~\ref{thm:main res}, 
\be
|I_{\mathrm{out}}(\la)|/I_{\mathrm{in}}(\la) \leq 5\frac{R^{2L}d^{L+1}}{\la^L}.\label{out}
\ee 
\end{lemma}
Lemma~\ref{lma:out} is proved in Section~\ref{sec:out}. The three lemmas conclude the proof of Theorem~\ref{thm:main res}. Indeed, they imply $I(\la)>0$ when $R^{2L}d^{L+1}/\la^L$ is sufficiently small, give the desired form for the coefficients, and imply
\be
\bigg|\log I(\la)-\sum_{k=1}^{L-1}b_k(f,g)\la^{-k}\bigg|\leq \bigg|\log I_{\mathrm{in}}(\la)-\sum_{k=1}^{L-1}b_k(f,g)\la^{-k}\bigg| + \left|\log\left(1+\frac{I_{\mathrm{out}}(\la)}{I_{\mathrm{in}}(\la)}\right)\right|.
\ee We now use~\eqref{in},~\eqref{out} to conclude~\eqref{main res}.

\section{Proof of Lemma~\ref{lma:in}}\label{sec:in}
Using Assumption~\ref{ass:fg}, part~\ref{Ial}, and using that $R\sqrt\e\leq r_0$, we have
\be
\max_{\|x\|\leq R\sqrt\e}\la|\mathcal{R}_{2L+2}(x)| \les \la(R^2\e)^{L+1}=R^{2L+2}d\e^L. 
\ee
Therefore, recalling~\eqref{eq:Taylor-f}, we have
\be\label{I-in-1}
I_{\mathrm{in}}(\la)=e^{\CO(R^{2L+2}d\e^L)}\left(\frac{\la}{2\pi}\right)^{d/2}\int_{\|x\|\leq R\sqrt\e}g(x)e^{-\la f_{2L+1}(x)}\dd x.
\ee
The integrand must be positive for~\eqref{I-in-1} to hold. But indeed, $g(0)=1$ and $\log g$ has bounded first derivative, both by~\ref{Ial g} of Assumption~\ref{ass:fg}. Thus we can ensure $g(x)>0$ for all $\|x\|\leq R\sqrt\e$ by choosing the upper bound $c^+$ on $R\sqrt\e$ small enough in~\eqref{R ineq v2}. This proves our claim in Lemma~\ref{lma:in} that $I_{\mathrm{in}}(\la)>0$.

\subsection{Iterative change of variables}\label{subsec:change}
Throughout this section, we assume all the conditions of Lemma~\ref{lma:in} hold true without explicitly saying so. Before stating a key change of variables, we introduce the following two definitions.

\begin{definition}\label{def:base comp}
A \emph{base} tensor $G_{k\to j}$ is a symmetric multilinear map as in Definition~\ref{def:Akj}, which \emph{only depends on $d$, $\nabla^\ell f(0)$, $\ell=3,\dots,2L+1$, and $\nabla^\ell(\log g)(0)$, $\ell=1,\dots,2L-1$,} but does not explicitly depend on $\la$, and satisfies $\|G_k\|_\op< c(\argc{f},\argc{g})$. Here $f$ and $g$ are the same as in \eqref{lap}. A \emph{composite} tensor $\ef{k\to j}$, denoted specifically by the letter $F$, is any tensor which can be written as
$$\ef{k\to j} = \sum_{\ell\geq0} \e^\ell G_{k\to j}^{(\ell)},$$ for base tensors $G_{k\to j}^{(\ell)}$. \end{definition}
Any time we write $\ef{k\to j}$ (or $\ef{k}$ when $j=0$) we mean a composite tensor according to this definition.

\begin{definition}\label{def:O}
We say $f(x)=\CO(\|x\|^p)$ for all $x\in\us\subset\{\|x\|\le 1\}$ if there exists $c$ depending only on $\argc{f},\argc{g}$ but independent of $d$ and $\la$, such that $|f(x)|\leq c\|x\|^p$ for all $x\in\us$. 
\end{definition} 


\begin{lemma}\label{lma:changevar}
Let $f_{2L+1}$ be as in~\eqref{eq:Taylor-f}. There exists $X(t)=t+\varphi(t)$ with $\varphi(t)=\sum_{q= 2}^{2L}\ef{q\to1}[t^{\otimes q}]$ such that $f_{2L+1}(X(t))=\frac12\|t\|^2 + \CO(\|t\|^{2L+2})$ for all $\|t\|\leq 2R\sqrt\e$. The function $\varphi$ is explicitly computable from the derivatives of $f$ of order $3,\dots,2L+1$.
\end{lemma}
See Appendix~\ref{app:change} for the proof and construction of $\varphi$. As an example, when $L=1$, we take $X(t)=t-\frac16\nabla^3f(0)[t^{\otimes2}]$, where $\nabla^3f(0)[t^{\otimes2}]$ is the vector such that $\nabla^3f(0)[t^{\otimes2}]^\top u=\nabla^3f(0)[t,t,u]$ for all $u\in\br^d$. It is then straightforward to check that 
\bs
f_3\left(t-\tfrac16\nabla^3f(0)[t^{\otimes2}]\right)&=\tfrac12\left\|t-\tfrac16\nabla^3f(0)[t^{\otimes2}]\right\|^2+\tfrac16\nabla^3f(0)\left[\left(t-\tfrac16\nabla^3f(0)[t^{\otimes2}]\right)^{\otimes3}\right]\\
&=\|t\|^2/2+ \CO(\|t\|^4).
\es The basic observation that substituting $x=t-F[t^{\otimes k-1}]$ into $\|x\|^2/2+F[x^{\otimes k}]$ kills the order $k$ polynomial is at the heart of all of our changes of variables.\\

We now show that $X$ is bijective and characterize the set $X^{-1}(\{\|x\|\leq R\sqrt\e\})$. The following lemma states a slightly more general result which will be needed later on.
\begin{lemma}\label{lma:U} 
Let $\varphi(t)=\sum_{q\geq 2}\ef{q\to1}[t^{\otimes q}]$, and $X(t)=t+\varphi(t)$. Fix any absolute constants $C_1,C_2$ such that $0<C_1\leq C_2$. For all $r\leq1/(2C_2)$ small enough that $\|\varphi'(t)\|\leq \frac12$ $\forall\|t\|\leq 2C_2r$, and for any set $\us$ satisfying $\{\|t\|\leq C_1r\}\subseteq\us\subseteq\{\|t\|\leq C_2r\}$, it holds
\begin{enumerate}[label=(\arabic*)]
\item  $\{\|t\|\leq\tfrac23 C_1r\}\subseteq X^{-1}(\us)\subseteq\{\|t\|\leq 2C_2r\}$, and 
\item $X$ is a bijection from $X^{-1}(\us)$ onto $\us$.
\end{enumerate}
\end{lemma}
See Appendix~\ref{app:change} for the proof. Let 
\be\label{usus1}
\us=\{\|x\|\leq R\sqrt\e\},\quad \us_1=X^{-1}(\{\|x\|\leq R\sqrt\e\}). 
\ee
Lemma~\ref{lma:U} with $C_1=C_2=1$,  $r=R\sqrt\e$, and $\us$ as above gives that if $R\sqrt\e$ is small enough then $\{\|t\|\leq \frac23R\sqrt\e\}\subseteq\us_1\subseteq\{\|t\|\leq 2R\sqrt\e\}$. Thus in particular, the conclusion of Lemma~\ref{lma:changevar} holds for all $t\in\us_1$. Combining~\eqref{I-in-1}, Lemma~\ref{lma:changevar}, and the fact that $X(t):X^{-1}(\us)\to\us$ is a bijection by Lemma~\ref{lma:U}, we have
\be\label{initchange}
I_{\mathrm{in}}(\la)=e^{\CO(R^{2L+2}d\e^L)}\left(\frac{\la}{2\pi}\right)^{d/2}\int_{\us_1}\exp\left(-\frac{\la}{2}\|t\|^2+\log g(X(t))+\log \detm(X'(t))\right)\dd t.
\ee
Since $\|\varphi'(t)\|_\op\ll 1$ for $t\in\us_1$, we have $\log\detm X'(t)=\log \detm( I_d+\varphi'(t))=\tr\log(I_d+\varphi'(t))\approx\tr\varphi'(t)\sim d$. This is made precise in the following lemma.


\begin{lemma}\label{lma:logdet exp}
We have
\bs\label{logdet v3}
&\log\det(X^\prime(t))=d\bigg[\sum_{k=1}^{2L-1}\ef{k}[t^{\otimes k}]+\CO\big(\norm{t}^{2L}\big)\bigg],\quad\forall t\in\us_1.
\es
\end{lemma}

The proof follows from Lemma~\ref{lma:logdet} with $m=0$, and the fact that $R\sqrt\e$ can be made sufficiently small by choosing $c^+$ appropriately in~\eqref{R ineq v2}. Next, we expand $\log g(X(t))$.

\begin{lemma}\label{lma:log-g}It holds
\bs\label{log-g}
&\log g(X(t))=\sum_{k=1}^{2L-1}\ef{k}[t^{\otimes k}]+\CO\big(\norm{t}^{2L}\big),\quad\forall t\in\us_1.
\es
\end{lemma}

\begin{proof}
Using~\eqref{eq:Taylor-g} and~\eqref{Gk norms}, and the fact that $\tfrac{1}{k!}\nabla^k(\log g)(0)$ is a base tensor and therefore a composite tensor, we have
\bs
\log g(t+\varphi(t))=\sum_{k=1}^{2L-1} \ef{k}[(t+\varphi(t))^{\otimes k}] + \mathcal{R}_{2L}^{(g)}(t+\varphi(t)),\quad\forall t\in\us_1.
\es 
Now, recall that $R\sqrt\e\leq r_0$ is one of the assumptions of Lemma~\ref{lma:in}. Since $\|t+\varphi(t)\|=\|X(t)\|\leq R\sqrt\e\leq r_0$ when $t\in\us_1$ (by the definition of $\us_1$), we have $|\mathcal{R}_{2L}^{(g)}(t+\varphi(t))|\les\|t+\varphi(t)\|^{2L}\les\|t\|^{2L}$. Here, we used that $\|\varphi(t)\|\les\|t\|$ for $t\in\us_1$.  Next, using Lemma~\ref{lma:changevar}, 
\be
t+\varphi(t)=t+\sum_{q\geq2}\ef{q\to1}[t^{\otimes q}]=\sum_{q\geq1}\ef{q\to1}[t^{\otimes q}].
\ee
But then \eqref{njouter} with $j=0$ in Corollary~\ref{corr:pq} gives $\sum_{k=1}^{2L-1} \ef{k}[(t+\varphi(t))^{\otimes k}]=\sum_{k\geq1}\ef{k}[t^{\otimes k}]$. Thus we have shown $\log g(t+\varphi(t))=\sum_{k\geq1}\ef{k}[t^{\otimes k}]+\CO(\|t\|^{2L})$, $t\in\us_1$. To conclude, we move the part of the sum with $k\geq 2L$ into the remainder $\CO(\|t\|^{2L})$.
\end{proof}

We now use~\eqref{logdet v3} and~\eqref{log-g} in~\eqref{initchange} to get
\bs
\frac{\la}{2}\|t\|^2&-\log g(X(t))-\log \detm( X'(t)) \\
&= \la\left(\frac12\|t\|^2+\frac1\la\sum_{k=1}^{2L-1}\ef{k}[t^{\otimes k}]+\e\sum_{k=1}^{2L-1}\ef{k}[t^{\otimes k}]\right)+\CO(\|t\|^{2L}),\quad\forall t\in\us_1.
\es
We can combine $\frac1\la\ef{k}$ and $\e\ef{k}$ into $\e\ef{k}$. Also, since $\|t\|\leq 2R\sqrt\e$ for all $t\in\us_1$ as discussed below~\eqref{usus1}, we have $\CO\big(\norm{t}^{2L}\big)=\CO(R^{2L}\e^L)$. We conclude that
\bs\label{E1}
I_{\mathrm{in}}(\la)&=e^{\CO(R^{2L+2}d\e^L)}\left(\frac{\la}{2\pi}\right)^{d/2}\int_{\us_1}\exp\left(-\la E_1(t)\right)\dd t,\\
E_1(t)&:=\tfrac12\norm{t}^2+\e\sum_{k=1}^{2L-1}\ef{k}[t^{\otimes k}],\\
\{\|t\|\leq \tfrac23R\sqrt\e\}&\subseteq\us_1\subseteq\{\|t\|\leq 2R\sqrt\e\}.
\es
Comparing $E_1$ in~\eqref{E1} with the original $f$, we see that $E_1$ is \emph{closer to being exactly quadratic}. If $L=1$, we stop here and estimate $I_{\mathrm{in}}(\la)$ as in Section~\ref{ssec:rem bnd}. If $L\ge 2$, we show next that by iteratively changing variables, we can continue to increase the power of $\e$ in front of cubic and higher powers of $t$.
%
 \begin{lemma}\label{lma:changevar-m}
 Let $1\leq m\leq L-1$ and
 \be\label{Em}
 E_m(t)=\e\ef{1}[t] + \e\ef{2}[t^{\otimes 2}]+\tfrac12\norm{t}^2+\e^m\sum_{k=3}^{2L-2m+1}\ef{k}[t^{\otimes k}].
 \ee 
Let $C$ be an absolute constant to be chosen later. There is an explicitly computable change of variables 
\be
T_m(s)=s+\e^m\varphi_m(s),\quad \varphi_m(s)=\sum_{k=2}^{2L-2m}\ef{k\to1}[s^{\otimes k}],
\ee 
such that for all $\|s\|\leq CR\sqrt\e$,
\be\label{Emt}
 E_m(T_m(s))=\e\ef{1}[s] + \e\ef{2}[s^{\otimes 2}]+\tfrac12\norm{s}^2+\e^{m+1}\sum_{k=3}^{2L-2m-1}\ef{k}[s^{\otimes k}]+\CO\big((R\sqrt\e)^{2L+2}\big).
 \ee  
\end{lemma}
See the end of Appendix~\ref{app:change} for the proof and the construction of $T_m$. The value of $C$ will be chosen below the proof of Corollary~\ref{corr:E}.

As we will see below, the big-$\CO$ term in~\eqref{Emt} contributes the factor $\exp\big(\CO(\la (R\sqrt\e)^{2L+2})\big)$ to $I_{\mathrm{in}}$. Observe that this matches the desired order of magnitude in Lemma~\ref{lma:in}.  Setting $m=1$ in \eqref{Em} gives the function $E_1$ in \eqref{E1}. 

Comparing~\eqref{Em} to~\eqref{Emt}, we see that the power of $\e$ and highest power $k$ of $s$ have changed. The effect of $T_m(s)$ is only to increase the power of $\e$ from $m$ to $m+1$, not to decrease the highest power of $s$ from $2L-2m+1$ to $2L-2m-1$. Terms with these higher powers of $s$ do appear upon plugging in $T_m(s)$ to $E_m$. However, precisely because of the higher power of $\e$, any $\e^{m+1}\ef{k}[s^{\otimes k}]$ with $k>2L-2m-1$ can simply be thrown out, i.e. absorbed into the $\CO((R\sqrt\e)^{2L+2})$ remainder.\\




In the following lemma, we study the change of variables $T_m$.

\begin{lemma}\label{lma:phi-m} Let $1\leq m\leq L-1$, $T_m$ be as in Lemma~\ref{lma:changevar-m} and $\us_m$ be a set satisfying
\be
\{\|t\|\leq c_mR\sqrt \e\}\subseteq \us_m\subseteq\{\|t\|\leq C_mR\sqrt \e\}
\ee
for some absolute constants $0<c_m\leq C_m$. Let $\us_{m+1}=T_m^{-1}(\us_m)$. For all $\e$ small enough that $\e^m\|\varphi_m'(t)\|\leq1/2$ for all $\|t\|\leq 2C_{m}R\sqrt\e$, it holds
\begin{enumerate}[label=(\arabic*)]
\item 
\be
\{\|t\|\leq \tfrac23c_{m}R\sqrt \e\}\subseteq \us_{m+1}\subseteq \{\|t\|\leq 2C_{m}R\sqrt\e\},
\ee 
\item $T_m:\us_{m+1}\to \us_m$ is bijective, and 
\item We have
\bs\label{logdet expr}
\log\detm(T_m'(s)) = d\e^m\sum_{k=1}^{2L-2m-1}\ef{k}[s^{\otimes k}] + \CO(dR^{2L}\e^{L})\quad\forall s\in\us_{m+1}.
\es
\end{enumerate}
\end{lemma}
The statements (1) and (2) follow from Lemma~\ref{lma:U} with $r=R\sqrt\e$. Part (3) follows by setting $N=2L-2m$ in Lemma~\ref{lma:logdet} and using part (1) to characterize the diameter of $\us_{m+1}$. We combine the above two lemmas in the following corollary, thereby completing one full iteration from $E_m$ to $E_{m+1}$.

\begin{corollary}\label{corr:E}Let $1\leq m\leq L-1$, $E_m,E_{m+1},\varphi_m$ be as in Lemma~\ref{lma:changevar-m} and $\us_m$ be as in Lemma~\ref{lma:phi-m}. Then
\bs\label{Em1}
\la E_m(T_m(s)) +\log\detm(T_m'(s)) = \la E_{m+1}(s) +\CO(R^{2L+2}d\e^L)\quad\forall s\in\us_{m+1}.
\es 
Therefore, 
\be\label{Emm1}
\int_{\us_m}e^{-\la E_m(t)}\dd t = e^{\CO(R^{2L+2}d\e^L)}\int_{\us_{m+1}}e^{-\la E_{m+1}(s)}\dd s.
\ee
\end{corollary}
\begin{proof}
We combine Lemma~\ref{lma:phi-m} and Lemma~\ref{lma:changevar-m} to study $E_m(T_m(s)) +\tfrac1\la\log\detm(T_m'(s))$. Clearly, $\frac1\la\times d\e^m=\e^{m+1}$. Then
\bs\label{Emi}
E_m(&T_m(s)) +\tfrac1\la\log\detm(T_m'(s))\\
=&\e\ef{1}[s] + \e\ef{2}[s^{\otimes 2}]+\tfrac12\norm{s}^2+\e^{m+1}\sum_{k=3}^{2L-2m-1}\ef{k}[s^{\otimes k}]\\
&+\e^{m+1}\sum_{k=1}^{2L-2m-1}\ef{k}[s^{\otimes k}]+\CO(R^{2L+2}\e^{L+1}) \\
=&E_{m+1}(s)+\CO(R^{2L+2}\e^{L+1}).
\es
Thus, comparing the second and third line in~\eqref{Emi}, we see that each polynomial term in $\tfrac1\la\log\detm(T_m'(s))$ has an equal or higher power of $\e$ than the polynomial term in $E_m(T_m(s))$ of the same degree. In other words, the log Jacobian does not harm the convenient structure created by changing variables.  
Multiplying both sides of~\eqref{Emi} by $\la$ and noting that $\la R^{2L+2}\e^{L+1}=dR^{2L+2}\e^L$ proves~\eqref{Em1}.

To prove~\eqref{Emm1}, we use the change of variables $t=T_m(s)$, which is a bijection from $\us_{m+1}$ onto $\us_m$ by Lemma~\ref{lma:phi-m}. We then apply~\eqref{Em1} to conclude.
\end{proof}

Starting with~\eqref{E1}, we iteratively apply Lemma~\ref{lma:phi-m} and Corollary~\ref{corr:E}, stopping once we get to $m+1=L$. We conclude that
\bs\label{IJ123}
I_{\mathrm{in}}(\la)&=e^{\CO(R^{2L+2}d\e^L)}\left(\frac{\la}{2\pi}\right)^{d/2}\int_{\us_{L}}e^{-\la E_{L}(t)}\dd t,\\
E_{L}(t)&= \e \ef{1}[t]+\e \ef{2}[t^{\otimes 2}]+\tfrac12\|t\|^2,\quad
\us_L=T_{L-1}^{-1}\circ \dots \circ T_1^{-1}(\us_1).
\es 
Furthermore, by the repeated application of Lemma~\ref{lma:phi-m}, and using that $c_1=2/3, C_1=2$  (recalling~\eqref{E1}), we have
\be\label{UL size}
\{\|t\|\leq (2/3)^LR\sqrt \e\}\subseteq \us_{L}\subseteq \{\|t\|\leq 2^LR\sqrt\e\}.
\ee
Thus we see that in Lemma~\ref{lma:changevar-m}, we can take $C=2^L$. Note that~\eqref{IJ123} and \eqref{UL size} both also hold for $L=1$, i.e. $E_L$ from~\eqref{IJ123} coincides in structure with $E_1$ from~\eqref{E1}, upon setting $\ef{2}=0$ in~\eqref{IJ123}. 

The following quantities will appear in Sections~\ref{ssec:rem bnd},~\ref{sec:terms}, and~\ref{sec:meas}. 
\begin{definition}\label{def:aB}Let $a,B$ be such that $E_L$ in~\eqref{IJ123} can be written as
\be\label{ELaB}E_L(t)=a^\top t + \frac12t^\top Bt.\ee Thus in particular,
\bs\label{aB}
a&=\e J_1,\quad B=I_d+2\e J_2,\\
J_1&=\ef{0\to1},\quad J_2=\ef{0\to2}.
\es where $\ef{0\to1}$ and $\ef{0\to2}$ are the vector and matrix identifications of the specific $\ef{1}$, $\ef{2}$ appearing in $E_L$ in~\eqref{IJ123}, respectively.
\end{definition}
Since $\ef{2}$ is by definition a symmetric bilinear form, the matrix $B$ is symmetric.
\subsection{Completing the square}\label{ssec:rem bnd}
Before presenting the main result of the section, we make the following observation: let $A=\ef{2}=\ef{2\to0}$ be a composite tensor given by a bilinear form, i.e. taking in two vectors and returning a scalar. 
Then $A$ can also be viewed as $A=\ef{0\to2}$, i.e. a constant mapping returning a matrix. The same is true for $\ef{1\to0}$ also being $\ef{0\to1}$. The operator norms are preserved under this identification.\\

The main result in this section is the following.
\begin{lemma}\label{lma:square} Under the assumptions of Lemma~\ref{lma:in}, it holds
\be\label{logIJ}
\log I_{\mathrm{in}}(\la)=\frac\la2a^\top B^{-1}a-\frac12\log\det B+\CO(R^{2L+2}d\e^L),
\ee where $a,B$ are as in~\eqref{aB}. 
\end{lemma}
\begin{proof}
Recall $B$ is symmetric, and it is invertible if $\e$ is small enough. By definition of $a,B$, we have
\bs\label{after sqr}
E_{L}(t)&= a^\top t + \frac12t^\top Bt = \frac12\|B^{1/2}t +B^{-1/2}a\|^2  - \frac12a^\top B^{-1}a.
\es
We now do the final change of variables
\be\label{tQs}
t=Q(s) = B^{-1/2}(s-B^{-1/2}a).
\ee
Then
\bs\label{eBa}
\left(\frac{\la}{2\pi}\right)^{d/2}\int_{\us_{L}}e^{-\la E_{L}(t)}\dd t&=e^{\frac\la2a^\top B^{-1}a}\left(\frac{\la}{2\pi}\right)^{d/2}\int_{\us_L}e^{-\frac\la2\|B^{1/2}t +B^{-1/2}a\|^2}\dd t\\
&=e^{\frac\la2a^\top B^{-1}a}(\det B)^{-1/2}\mathbb P(Z\in\sqrt\la\tilde\us_L),
\es 
where $Z\sim\mathcal N(0, I_d)$ and $\tilde\us_L=B^{1/2}\us_{L}+B^{-1/2}a$. Note that $\|B^{-1}\|_\op\leq (1-2\e\|\ef{2}\|_\op)^{-1}$, which is bounded if $\e$ is sufficiently small. $\|B\|_\op$ is also bounded, and $\|a\|\les\e$. Recall from~\eqref{UL size} that $\{\|t\|\leq CR\sqrt\e\}\subseteq\us_L$ for some absolute constant $C$. We find a $C'=C'(\argc{f},\argc{g})>0$ such that $\{\|x\|\leq C'R\sqrt\e\}\subset\tilde\us_L$. To do so, it suffices to prove that if $x=B^{1/2}t+B^{-1/2}a$ and $\|x\|\leq C'R\sqrt\e$ then $\|t\|\leq CR\sqrt\e$. Indeed, we have $t=B^{-1/2}x-B^{-1}a$, and therefore $\|t\|\leq \argc{f,g}(C'R\sqrt\e +\e)$. Here, $\argc{f,g}$ is some constant depending on $\argc{f},\argc{g}$, only. Thus, $C'$ can be found from the inequality
\be
\argc{f,g}(C'R\sqrt\e +\e)\le CR\sqrt\e
\ee
for all $\e$ sufficiently small. Dividing by $R\sqrt\e $ and using that $R\ge1$ and $\e$ can be made as small as necessary, by \eqref{R ineq v1}, we see that $C'$ can be found.

Thus $\{\|x\|\leq C'R\sqrt d\}\subset\sqrt\la\tilde\us_L$. Assuming furthermore that $C'R\geq2$ by~
\eqref{R ineq v2}, we have
$$
1-\exp(-{C'}^2R^2d/8)\leq 1-\exp(-(C'R-1)^2d/2)\leq\mathbb P(Z\in\sqrt\la\tilde\us_L)\leq 1.
$$ Finally, by~\eqref{R ineq v2} we can also assume $R^2\geq \frac{8L\log\la}{{C'}^2d}$. We then obtain the following further lower bound:
\be\label{PZlb}
1-\la^{-L}\leq  \mathbb P(Z\in\sqrt\la\tilde\us_L)\leq 1.
\ee 
Combining~\eqref{IJ123},~\eqref{eBa}, and~\eqref{PZlb} now gives 
\be
I_{\mathrm{in}}(\la)=e^{\CO(R^{2L+2}d\e^L)}e^{\frac\la2a^\top B^{-1}a}(\det B)^{-1/2}.\ee
This proves~\eqref{logIJ}.\end{proof}

\subsection{Structure of terms in the expansion}\label{sec:terms}
Recall from~\eqref{aB} that $a=\e J_1$ and $B=I_d+2\e J_2$, and consider the two terms in~\eqref{logIJ}. 
We have $\log\det(I_d+2\e J_2)=\tr\log(I_d+2\e J_2)$. Using that $\|2\e J_2\|_\op <1/2$ if $\e$ is sufficiently small (since $J_2=\ef{0\to2}$), we have
\bs
\left\|(I_d+2\e J_2)^{-1}- \sum_{k=0}^{L-2}  (- 2\e J_2)^{k}\right\|_\op&\les \e^{L-1},\\
\left\|\log(I_d+2\e J_2) + \sum_{k=1}^{L-1} \tfrac1k(-2\e J_2)^k\right\|_\op&\les \e^L.
\es 
By Lemma~\ref{lma:op} (see also Remark~\ref{rk:matmul}) the first sum is $\ef{0\to2}$ and the second sum is $\e\ef{0\to2}$. Therefore, 
\bs
J_1^\top(I_d+2\e J_2)^{-1}J_1=&J_1^\top\ef{0\to2}J_1+\CO(\e^{L-1}), \\
\tr\log(I_d+2\e J_2)=&\e\tr\ef{0\to2} +\CO(d\e^L). 
\es
From the Definition~\ref{def:base comp} and since $J_1=\ef{0\to1}$, it is clear that $J_1^\top\ef{0\to2}J_1=\ef{0\to0}=:\mu$, where $\mu$ is just a polynomial in $\e$, i.e. $\mu=\sum_{k\ge0}\mu_k\e^k$. Its coefficients are bounded and depend only on $d$, $\nabla^kf(0)$, and $\nabla^k(\log g)(0)$. We conclude that
\bs\label{lae}
\tfrac12\la\e^2J_1^\top&(I_d+2\e J_2)^{-1}J_1-\tfrac12\log\det (I_d+2\e J_2)\\
&=\la\e^2\mu +\e\tr\ef{0\to2} +\CO(d\e^{L}).
\es
Finally, 
write $\ef{0\to2}=\sum_{\ell\geq0}\e^\ell B_\ell$, for bounded matrices $B_\ell$ depending only on $d$, $\nabla^kf(0)$, and $\nabla^k(\log g)(0)$. Let $a_k=d^{-1}\tr B_{k-1}$, which are bounded. Then
\bs\label{laef}
\la\e^2\mu +\e\tr\ef{0\to2}&=\la\e^2 \mu +\e\sum_{\ell\geq0}\e^{\ell}\tr(B_\ell)+\CO(d\e^{L})\\
&=\la\sum_{k=2}^L \mu_k\e^k +\CO(\la\e^{L+1})+\e\sum_{\ell=0}^{L-2}\e^{\ell}\tr(B_\ell)+\CO(d\e^L)\\
&=\la\sum_{k=2}^L\mu_k\e^k+\sum_{k=1}^{L-1}da_k\e^k+\CO(d\e^{L})=\sum_{k=1}^{L-1}(\mu_{k+1}+a_k)d\e^k+\CO(d\e^{L}).
\es 
To get the second line, we moved the part of the polynomial $\mu=\sum_{k\ge0}\mu_k\e^k$ in which $k\geq L-1$ into the remainder. Similarly, we moved the part of the second sum in which $\ell\geq L-1$ into the remainder. 

Let $b_k(f,g,d)=(\mu_{k+1}+a_k)d^{k+1}$, so that $|b_k(f,g,d)|\les d^{k+1}$. Combining~\eqref{lae} with~\eqref{laef} and using the result in~\eqref{logIJ}, we finally have
\bs
\log I_{\mathrm{in}}(\la)=\sum_{k=1}^{L-1}b_k(f,g,d)\la^{-k}+\CO(R^{2L+2}d\e^L)
\es
This concludes the proof of Lemma~\ref{lma:in}. 
\section{Cumulants and proof of Lemma~\ref{prop:ad}}\label{sec:cum}

In this section, we prove Lemma~\ref{prop:ad}. First, we introduce the concept of cumulants. 
\subsection{Preliminaries on cumulants}
\begin{definition}[Cumulants]\label{def:cum}Let $Y_1,\dots, Y_m\in\br$ be random variables. We define
\be
\cum(Y_1,\dots,Y_m)=(-i)^m\pa_{s_1}\dots\pa_{s_m}\log\E\left[\exp\left(i\left[s_1Y_1+\dots+s_mY_m\right]\right)\right]\big\vert_{s_1=\dots=s_m=0}.
\ee 
See~\cite[Chapter 3.1]{peccati2011}.
\end{definition}
\begin{remark}\label{rk:icum}
    Under suitable integrability conditions (e.g. if $Y_1,\dots,Y_m$ are bounded random variables), the following definition is equivalent: $$\cum(Y_1,\dots,Y_m)=\pa_{s_1}\dots\pa_{s_m}\log\E\left[\exp\left(s_1Y_1+\dots+s_mY_m\right)\right]\big\vert_{s_1=\dots=s_m=0}.$$
\end{remark}

    When there is repetition among the $Y_j$, there is an alternative expression for the cumulant. Specifically, suppose we have another set of random variables $X_k$, $k=1,2,3,\dots$. Let $\alpha=(\alpha_1,\dots,\alpha_M)$ with $\alpha_j\geq0$. Suppose that $Y_1,\dots, Y_m$ consist of $\alpha_1$ copies of $X_1$, $\alpha_2$ copies of $X_2$, and so on, up to $\alpha_M$ copies of $X_M$. Let $|\alpha|=\alpha_1+\dots+\alpha_M=m$. Then 
\bs\label{cum id}
\cum(Y_1,\dots,Y_m)&=\cum\big(\underbrace{X_1,\dots,X_1}_{\alpha_1},\underbrace{X_2,\dots,X_2}_{\alpha_2},\dots,\underbrace{X_M,\dots,X_M}_{\alpha_M}\big)\\
&=(-i)^{|\alpha|}\pa_{u_1}^{\alpha_1}\dots\pa_{u_M}^{\alpha_M} \log\E\left[\exp\left(i[u_1X_1+\dots+u_MX_M]\right)\right]\big\vert_{u_1=\dots=u_M=0}.
    \es This follows from the fact that, for a function $f(s_1,s_2)=g(s_1+s_2)$, it holds $\pa_{s_1}\pa_{s_2}f(0,0) =g''(0)$.
We will only explicitly use cumulants of order one and two, for which it holds
\bs\label{cum12}
\cum(X)&=\E[X],\\
\cum(X,Y)&=\Cov(X,Y)=\E[XY]-\E[X]\E[Y].\es
See~\cite[Example 3.2.3]{peccati2011}. Let 
\be\label{pkdef}
p_k(x)=\nabla^k(\log g)(0)[x^{\otimes k}] -\frac{1}{(k+1)(k+2)}\nabla^{k+2}f(0)[x^{\otimes k+2}],\quad k=1,2,3,\dots,
\ee 
be functions on $\br^d$. 
For example, for $p_1,p_2$, we have
\bs\label{p1p2}
p_1(x) &= \nabla g(0)[x] -\frac16\nabla^3f(0)[x^{\otimes3}],\\
p_2(x) &= \big(\nabla^2 g(0)-\nabla g(0)^{\otimes2}\big)[x^{\otimes2}] -\frac1{12}\nabla^4f(0)[ x^{\otimes4}].
\es

Let $X\in\br^d$ be a random variable. Recall that $\al$ is a multiindex. 
We define
\bs\label{cumpalph}
\cum(p_\alpha(X))&=\cum\big(\underbrace{p_1(X),\dots,p_1(X)}_{\alpha_1},\underbrace{p_2(X),\dots,p_2(X)}_{\alpha_2},\dots,\underbrace{p_M(X),\dots,p_M(X)}_{\alpha_M}\big)\\
&=(-i)^{|\alpha|}\pa_{u}^{\alpha}\log\E\left[\exp\left(i[u_1p_1(X)+\dots+u_Mp_M(X)]\right)\right]\big\vert_{u_1=\dots=u_M=0}.
\es 
Here, the second line is by \eqref{cum id}. Thus, for example, 
\be
\cum(p_{3,1,0,1}(X))=\cum(p_1(X),p_1(X),p_1(X),p_2(X),p_4(X)).
\ee 
For convenience, we recall the formula in Lemma~\ref{prop:ad}, which we will prove in Section~\ref{subsec:cum}.
\be\label{eq:c}
b_{\frac M2}(f,g,d)=b_{\frac M2}(f,g)=\sum_{\substack{\alpha_1,\dots,\alpha_M\geq0\\\sum_{i=1}^Mi\alpha_i=M}}\frac{\cum(p_\alpha(Z))}{\prod_{i=1}^M\alpha_i!(i!)^{\alpha_i}}.\ee The fact that there is no explicit dependence on $d$ is clear from the righthand formula. Indeed, we see from~\eqref{cumpalph} that in $\cum(p_\alpha(Z))$, the only appearance of $d$ is in the functions $p_1(Z),\dots,p_M(Z)$. But we see in~\eqref{pkdef} that the definition of $p_1,\dots,p_M$ is agnostic to the number of arguments the functions $f$ and $g$ have. 
%

\subsection{More on the terms $b_k$}\label{sec:c1}
We first study $b_{M/2}$ to determine the highest-order $f$ derivative contributing to it. This is useful in Section~\ref{sec:meas} below, where we study the derivative order of $b_{L-1}(f,g)-b_{L-1}(f,1)$. We then compute $b_1(f,g)$ explicitly. 

By~\eqref{pkdef}, the highest-order $f$ derivatives in the formula~\eqref{eq:c} for $b_{M/2}$ necessarily arise from $\alpha$ such that $\alpha_M>0$. But since $\sum_{i=1}^Mi\alpha_i=M$, if $\alpha_M>0$ then we must have $\alpha_M=1$ and $\alpha_i=0$ for all $i\neq M$. Thus the highest-order $f$ derivative contribution to $b_{M/2}(f,g)$ is 
\bs\label{hi-f-b}\frac{\cum(p_{0,\dots,0,1}(Z))}{M!}&=\frac{\cum(p_{M}(Z))}{M!}=\frac{\E[p_{M}(Z)]}{M!} \\
&=\frac{\E[\nabla^M(\log g)(0)[Z^{\otimes M}]]}{M!}-\frac{\E[\nabla^{M+2}f(0)[Z^{\otimes M+2}]]}{(M+2)!}.\es

Next, we use~\eqref{eq:c} with $M=2$ to compute $b_1(f,g)$. Only $\alpha=(2,0)$ and $\alpha=(0,1)$ satisfy $\alpha_1+2\alpha_2=2$. Thus
\be\label{c1fg}
b_1(f,g)=\frac{\cum(p_1(Z),p_1(Z))}{2!(1!)^{2}}+ \frac{\cum(p_2(Z))}{1!(2!)^{1}}= \frac12\Var(p_1(Z)) + \frac12\E[p_2(Z)],
\ee using~\eqref{cum12}. Recall $p_1,p_2$ from~\eqref{p1p2}. Let
\be
T_1=\nabla g(0),\quad T_2 =\nabla^2g(0)-\nabla g(0)^{\otimes2},\quad T_3=-\nabla^3f(0)/6,\quad T_4=-\nabla^4f(0)/12.\ee Then $p_1(x)=T_1[x]+T_3[x^{\otimes3}]$ and $p_2(x)=T_2[x^{\otimes2}]+T_4[x^{\otimes4}]$. The polynomial $p_1$ is odd, therefore $\E[p_1(Z)]=0$ and $\Var(p_1(Z))=\E[p_1(Z)^2]$. We now rewrite $p_1$ using Hermite polynomials. For $x=(x_1,\dots,x_d)$ define the first- and third-order multivariate Hermite polynomials~\cite[Chapter 1.3]{bogachev1998gaussian}
$$
H_i(x)=x_i,\qquad 
H_{ijk}(x)
=x_i x_j x_k - x_i\ind\{j=k\} - x_j\ind\{i=k\} - x_k\ind\{i=j\} .
$$
Thus $H_{iii}(x)=x_i^3-3x_i$, $H_{iij}(x)=(x_i^2-1)x_j$ for $i\neq j$, and $H_{ijk}(x)=x_i x_j x_k$ when $i,j,k$ are distinct. Reordering indices in the subscript does not change the polynomial. It is straightforward to check that we may write
\be
p_1(x) = \sum_{i=1}^d\bigg(T_1^i+3\sum_{j=1}^dT_3^{ijj}\bigg)H_i(x) + \sum_{i,j,k=1}^dT_3^{ijk}H_{ijk}(x).
\ee 
The Hermite polynomials are orthogonal in $L^2(\mathcal N(0,I_d))$~\cite[Chapter 1.3]{bogachev1998gaussian}. In other words, if $Z\sim\mathcal N(0, I_d)$, then $\E[H_i(Z)H_j(Z)]=0$ if $i\neq j$, $\E[H_i(Z)H_{jk\ell}(Z)]=0$ for any $i,j,k,\ell$, and $\E[H_{ijk}(Z)H_{\ell mn}(Z)]=0$ if $(i,j,k)\neq (\ell,m,n)$, viewed as unordered triplets. Furthermore, we have $\E[H(Z_i)^2]=1$ and it is straightforward to show $\E[H_{ijk}(Z)^2]=1=6/3!$ if $i,j,k$ are distinct, $\E[H_{iij}(Z)^2]=2=6/3$ if $i,j$ are distinct, and $\E[H_{iii}(Z)^2]=6=6/1$. Thus for general $i,j,k$, the expectation $\E[H_{ijk}(Z)^2]$ is given by $6$ divided by the number of distinct ways to rearrange the indices $i,j,k$.

Using these facts, and the symmetry of the tensor $T_3$, we conclude that
\bs\label{p1}
\Var(p_1(Z))&=\E[p_1(Z)^2] = \sum_{i=1}^d\bigg(T_1^i+3\sum_{j=1}^dT_3^{ijj}\bigg)^2 +6\|T_3\|_F^2\\
=&\|\nabla g(0)-\tfrac12\nabla\Delta f(0)\|^2+\tfrac16\|\nabla^3f(0)\|_F^2.
\es
Next, a straightforward calculation gives
\bs\label{p2}
\E[p_2(Z)] = &\sum_{i=1}^dT_2^{ii} +3\sum_{i,j=1}^dT_4^{iijj}
=\Delta g(0)-\|\nabla g(0)\|^2 -\frac14\Delta^2f(0).
\es
We substitute~\eqref{p1} and~\eqref{p2} in~\eqref{c1fg} to get
\bs\label{b1fg}
b_1(f,g) =&-\frac12\nabla\Delta f(0)^\top\nabla g(0)+ \frac18\|\nabla\Delta f(0)\|^2+\frac1{12}\|\nabla^3f(0)\|_F^2+\frac12\Delta g(0)-\frac18\Delta^2f(0).
\es
\subsection{Proof of Lemma~\ref{prop:ad}}\label{subsec:cum}
Let $t=\la^{-1/2}$, and define
\be\label{I0 t}
I_0(f, g, d, t)=\log I_{\mathrm{in}}(\la) = \log\left\{\Big(\frac{t^{-2}}{2\pi}\Big)^{d/2}\int_{\|x\|\leq t\log(1/t)\sqrt d} g(x)e^{-t^{-2} f(x)}\,\dd x\right\}.
\ee Here, note that we have chosen $R=R(t)=\log(1/t)$.
Let $c^\pm$ be as in
Lemma~\ref{lma:in}. Then this lemma gives that for all $f,g$ satisfying parts~\ref{Ial},~\ref{Ial g} of Assumption~\ref{ass:fg}, and if
\bs\label{Rtd}
R=&\log(1/t)\geq c^-\sqrt{1\vee\frac{2\log(1/t)}{d}},\\
R\sqrt{d}t =&\sqrt d\log(1/t)t \leq \min(c^+,r_0),\es
then
\be\label{Gbt}
\left|I_0(f,g,d,t)- \sum_{k=1}^{L-1}b_k(f,g,d)t^{2k}\right|\leq C(\argc{f},\argc{g}, d)t^{2L}\log^{2L+2}(1/t).
\ee Here, $C(\argc{f},\argc{g},d)$ is a constant depending only on $\argc{f},\argc{g},d$. 

By changing variables as $y=x/t$ in \eqref{I0 t}, it is easy to see that $I_0(f, g, d, t)$ is a smooth function of $t$ in a neighborhood of $t=0$.

To compute $b_k(f,g,d)$, we consider any arbitrary fixed $f,g,d$ such that parts~\ref{Ial},~\ref{Ial g} of Assumption~\ref{ass:fg} are satisfied, and take $t\to0$. The condition~\eqref{Rtd} is satisfied for all $t$ small enough, so~\eqref{Gbt} implies that $b_k(f,g,d)=\frac{1}{(2k)!}\pa_t^{2k}I_0(f,g,d,t)\vert_{t=0}$ for all $k=1,\dots,L-1$. Here, we have used that $t^{2L}\log^{2L+2}(1/t)=o(t^{2L-1})$, and $\pa_t^{2k}$ means the partial derivative with respect to the fourth argument, keeping the first three frozen at fixed values. More precisely, if $f$ and $g$ depend on $\la$, that $\la$ is held fixed when computing the derivatives.
Now that we have this expression for $b_k$, we are free to substitute any $f,g,d$ for which~\eqref{Gbt} is applicable, including $\la$-dependent $f,g,d$.

Let $I_0(t)$ be shorthand for $I_0(f,g,d,t)$ for a fixed $f,g,d$. We have $b_{M/2}(f,g,d)=I_0^{(M)}(0)/(M)!$. To compute this derivative, we modify $I_0$ to create new functions $I_1,I_2$, each of which differs from the previous one by $o(t^M)$. We will then show that $I_2(t) = \sum_{k=1}^M \tilde b_kt^{k}+o(t^M)$ for explicit $\tilde b_k$. This implies $b_{M/2}=\tilde b_M$.\\

Define the set
\be
A=\{z\in\br^d\,:\,\|z\|\leq\log(1/t)\sqrt d \}.
\ee 
Using Assumption~\ref{ass:fg}, we have that $g(tz)>0$ for all $z\in A$ provided $t$ is small enough, since $t\log(1/t)\sqrt d\to0$ as $t\to0$. We can therefore define
\be
F_g(t,z)=\log g(tz) -\left(t^{-2}f(tz)-\|z\|^2/2\right),\quad z\in A.
\ee for $t$ small enough. Thus $I_0(t)=\log\E[ \exp( F_g(t,Z))\ind_A(Z)]$ for a standard Gaussian $Z$ in $\br^d$. We now replace $F_g$ by its Taylor expansion in $t$ about $t=0$.

\begin{lemma}\label{lma:G1G2} Let the $p_k$ be as in~\eqref{pkdef} and define
\bs
I_1(t)=\log\E\left[ \exp(\tilde F_g(t,Z))\ind_A(Z)\right],\quad \tilde F_g(t,z):= \sum_{k=1}^{M}\frac{t^{k}}{k!}p_k(z),
\es 
for $M\leq 2L-1$. Then $(I_0-I_1)(t)=o(t^{M})$.
\end{lemma}

Here and below in this section, the constant factors absorbed in small-$o$, big-$\CO$, and $\lesssim$ may depend on any parameter other than $t$. Note that $\tilde F_g(t,z)$ is precisely the $M$th Taylor polynomial of $F_g(t,z)$ in $t$ at $t=0$, and this is how the $p_k$ are constructed.

\begin{proof}
Let $F_g$ and $\tilde F_g$ be shorthand for $F_g(t,Z)$ and $\tilde F_g(t,Z)$, respectively. We have $I_0(t)-I_1(t)=\log(1+\delta(t))$, where
\be
\delta(t)=\frac{\E[e^{\tilde F_g}(e^{F_g-\tilde F_g}-1)\ind_A]}{\E[e^{\tilde F_g}\ind_A]}.
\ee
Using parts~\ref{Ial},~\ref{Ial g} of Assumption~\ref{ass:fg} 
and the definition \eqref{pkdef}, we have $t^k|p_k(z)|\lesssim t^k\log^{k+2}(1/t)$ and $|\mathcal{R}_{2L}^{(g)}(tz)+t^{-2}\mathcal{R}_{2L+2}(tz)|\lesssim t^{2L}\log^{2L+2}(1/t)$ for all $\|z\|\leq \log(1/t)\sqrt d$, provided $t$ is small enough. Thus
\bs
\sup_{\norm{z}\leq \log(1/t)\sqrt d}|&F_g(t,z)-\tilde F_g(t,z)|=\sup_{\norm{z}\leq \log(1/t)\sqrt d}\left|\sum_{k=M+1}^{2L-1}\frac{t^k}{k!}p_k(z) +\mathcal{R}_{2L}^{(g)}(tz)+t^{-2}\mathcal{R}_{2(L+1)}(tz)\right| \\
&\lesssim \sum_{k=M+1}^{2L}t^k\log^{k+2}(1/t)\lesssim t^{M+1}\log^{2L+2}(1/t)=o(t^M).\es
Similarly, 
\be
\sup_{\norm{z}\leq \log(1/t)\sqrt d}|\tilde F_g(t,z)|\lesssim t\log^{M+2}(1/t)\lesssim 1
\ee
for all $t$ small enough. Thus 
\bs
|\delta(t)|\leq \frac{e^{C(M)}\left|\exp\left(o(t^M)\right)-1\right|}{e^{-C(M)}}=o(t^M).
\es
Since $I_0(t)-I_1(t)=\log(1+\delta(t))$, we conclude $|I_0(t)-I_1(t)|= o(t^M)$ as well.
\end{proof}

\begin{lemma}\label{lma:G2G3}
Let $X(t)$ be the random vector given by the truncation of $Z\sim\mathcal N(0, I_d)$ to the region $\{\norm{x}\leq \log(1/t)\sqrt d\}$. Let $I_2(t)= \log\E[\exp(\tilde F_g(t, X(t)))]$. Then $(I_1-I_2)(t)=o(t^M)$.
\end{lemma} 
\begin{proof}
Note that $I_2(t)=I_1(t)-\log\P(A)$. Thus, by the standard Gaussian concentration,
\bs
|I_2(t)-I_1(t)| &= \left|\log\left(1-\P(A)\right)\right|\leq 2\P(A)\\
&\leq 2\exp(-(\log(1/t)-1)^2d/2)\leq 2\exp(-C\log^2(1/t)) =o(t^M)
\es 
for every $M$.
\end{proof}

\begin{lemma}\label{lma:G3}
It holds $I_2(t)=\sum_{m=1}^Mt^m\sum_{\alpha:\sum_ii\alpha_i=m}\frac{\cum(p_\alpha(Z))}{\prod_{i=1}^M\alpha_i!(i!)^{\alpha_i}} +o(t^{M})$. 
\end{lemma}

Combining Lemmas~\ref{lma:G1G2} and~\ref{lma:G2G3} gives $I_0(t)=I_2(t)+o(t^M)$. This implies that the coefficients in front of $t^k$, $k=1,\dots,M$, in~\eqref{Gbt} coincide with the corresponding coefficients in Lemma~\ref{lma:G3}. This concludes the proof of Proposition~\ref{prop:ad}.

To prove Lemma~\ref{lma:G3}, we need an auxiliary result.
\begin{lemma}\label{lma:XZ}
For each fixed $\alpha=(\alpha_1,\dots,\alpha_M)$, with $\alpha_j\geq0$, we have 
\be
|\cum(p_\alpha(X(t)))-\cum(p_\alpha(Z))| =o(t^M).
\ee
\end{lemma} 
See the end of the section for the proof of Lemma~\ref{lma:XZ}.

\begin{proof}[Proof of Lemma~\ref{lma:G3}] 
Write $I_2$  as $I_2(t)=H(t;t,t^2/2!,\dots,t^M/M!)$, where
$$
H(t;u_1,\dots,u_M)=\log\E\exp[u_1p_1(X(t))+\dots+u_Mp_M(X(t))].
$$ 
The function $H$ is $C^\infty$ in $u_1,\dots,u_M$. We Taylor expand $H$ to order $M$ in $u_1,\dots, u_M$. By the definition~\eqref{cumpalph} and Remark~\ref{rk:icum},
\be
\pa_{u_1}^{\alpha_1}\dots\pa_{u_M}^{\alpha_M} H(t;0)=\cum(p_\alpha(X(t))).
\ee 
Thus the Taylor expansion of $H$ takes the form
$$
H(t;u_1,\dots,u_M)=\sum_{1\leq|\alpha|\leq M}\cum(p_\alpha(X(t)))\prod_{k=1}^Mu_k^{\alpha_k}/\alpha_k! +\mathcal O(\|u\|^{M+1}).
$$ 
Substituting $u_k=t^k/k!$ gives
\bs\label{G3deriv}
I_2(t)&=H(t;t,t^2/2!,\dots,t^M/M!)\\
&=\sum_{1\leq|\alpha|\leq M}\cum(p_\alpha(X(t)))\prod_{k=1}^M(t^k/k!)^{\alpha_k}\frac{1}{\alpha_k!} +o(t^{M})\\
&=\sum_{m=1}^Mt^m\sum_{\alpha:\sum_kk\alpha_k=m}\frac{\cum(p^\alpha(X(t)))}{\prod_{k=1}^M\alpha_k!(k!)^{\alpha_k}} +o(t^{M})\\
&=\sum_{m=1}^Mt^m\sum_{\alpha:\sum_kk\alpha_k=m}\frac{\cum(p^\alpha(Z))}{\prod_{k=1}^M\alpha_k!(k!)^{\alpha_k}} +o(t^{M}).
\es 
To get the third line, we grouped like powers of $t$. To get the fourth line, we used Lemma~\ref{lma:XZ}. This concludes the proof. 
\end{proof}

The third line of~\eqref{G3deriv} gives the first $M$ terms in the Taylor series expansion of 
\be
\log\E[\exp(tp_1(X)+t^2p_2(X)/2!+\dots+t^Mp_M(X)/M!)], 
\ee
which is well-defined. Thus the fourth line of~\eqref{G3deriv} gives the first $M$ terms in the \emph{formal} power series expansion of $\log\E[\exp(tp_1(Z)+t^2p_2(Z)/2!+\dots+t^Mp_M(Z)/M!)]$, as claimed in Remark~\ref{rk:Zpower}.
\begin{proof}[Proof of Lemma~\ref{lma:XZ}]
Since cumulants are multilinear, and since the coefficients of the $p_k$'s are entries of the derivative tensors $F_{k+2}$, which are uniformly bounded, it suffices to show that
$|\cum(X(t)^{b_1},\dots,X(t)^{b_n})-\cum(Z^{b_1},\dots,Z^{b_n})|=o(t^M)$ for $n$ arbitrary multi-indices $b_i=(b_i^1,\dots,b_i^d)$, $i=1,\dots,n$. Here, $Z^{b_i}=Z_1^{b_i^1}\dots Z_d^{b_i^d}$, and similarly for $X(t)^{b_i}$. But $\cum(X(t)^{b_1},\dots,X(t)^{b_n})$ is a polynomial function of moments $\E[X(t)^b]$~\cite[(3.2.11)]{peccati2011}. Thus it suffices to bound $|\E[X(t)^b]-\E[Z^b]|$ for any multi-index $b$. We have
\bs
\E[X(t)^b]-\E[Z^b]&=\P(A)^{-1}\E[Z^b\ind_A]-\E[Z^b]\\
&=(\E[Z^b]\P(A^c)-\E[Z^b\ind_{A^c}])/\P(A).
\es 
Since $\P(A)\geq1/2$ (for $t$ sufficiently small) and $|\E[Z^b]|\leq \E[\|Z\|^{|b|}]\leq C(d, |b|)$, and using Cauchy-Schwarz, we have
\bs
\big|\E[X(t)^b]-\E[Z^b]\big|&\leq C(d,|b|)\P(A^c)^{1/2} \leq C(d,|b|)\exp(-(R(t)-1)^2d/4)\\
&\leq \exp(-C\log^2(1/t)),
\es recalling $R(t)=\log(1/t)$. As before, $\exp(-C\log^2(1/t))=o(t^{M})$ for every fixed $M$. Thus $|\E[X^b]-\E[Z^b]|=o(t^{M})$ for every $M$, and therefore $|\cum(X(t)^{b_1},\dots,X(t)^{b_n})-\cum(Z^{b_1},\dots,Z^{b_n})|=o(t^M)$ as well. Thus, $|\cum(p_\alpha(X(t)))-\cum(p_\alpha(Z))| =o(t^M)$ for each $\alpha$. 
\end{proof}

\section{Proof of Lemma~\ref{lma:out}}\label{sec:out}
We prove an upper bound on $|I_{\mathrm{out}}(\la)|$ and a lower bound on $I_{\mathrm{in}}(\la)$. 
\paragraph*{Lower bound on $I_{\mathrm{in}}(\la)$.} Assume $R\sqrt\e\leq r_0$. From \eqref{eq:Taylor-f} and \eqref{Fk norms} it follows that
\bs\label{f upbnd}
f(x)&\le \Big(1+\tfrac13\sup_{\|u\|\leq r_0}\|\nabla^3f(u)\|_\op R\sqrt\e\Big)\|x\|^2/2\\
&\leq(1+\argc{f}R\sqrt\e)\norm{x}^2/2,\quad \norm{x}\le R\sqrt\e.
\es We let $\mu=1+\argc{f}
R\sqrt\e$. Also, since $g(0)=1$ and $\log g$ has bounded first derivative for $\|x\|\leq R\sqrt\e<r_0$ by~\ref{Ial g} of Assumption~\ref{ass:fg}, we can choose $R\sqrt\e$ small enough that $g(x)\geq0.8$ for all $\|x\|\leq R\sqrt\e$. (Recall that the second inequality of~\eqref{R ineq v1} allows us to choose $R$ small enough.) Thus,
\be\label{inlb}
I_{\mathrm{in}}(\la)\geq 0.8\Big(\frac{\la}{2\pi}\Big)^{d/2}\int_{\|x\|\leq R\sqrt\e} e^{-\la \mu\norm{x}^2/2}\,\dd x=0.8\mu^{-d/2}\mathbb P\big(\|Z\|\leq R\sqrt{\mu}\sqrt d\,\big)\geq \frac12 \mu^{-d/2}.
\ee  To get the last inequality we used that $R\sqrt{\mu}>2$ and therefore $\mathbb P(\|Z\|\leq R\sqrt{\mu}\sqrt d)\geq \mathbb P\big(\|Z\|\leq 2\sqrt d\,\big)\geq 1-\exp(-d/2)\geq 0.63$ when $d\geq2L\geq2$.
\paragraph*{Upper bound on $|I_{\mathrm{out}}(\la)|$.}
Let
\bs
T_1&=\Big(\frac{\la}{2\pi}\Big)^{d/2}\int_{\br^d} (1+\|x\|)^{-\kappa\la d^{-p}}\dd x,\\
T_2&=\Big(\frac{\la}{2\pi}\Big)^{d/2}\int_{\norm{x}\ge R\sqrt\e} e^{-\kappa\sqrt{\la d}\|x\|}\,\dd x,\\
T_3&=\Big(\frac{\la}{2\pi}\Big)^{d/2}\int_{\norm{x}\ge R\sqrt\e} e^{-\la\kappa\|x\|^2/2}\,\dd x.
\es
Using~\eqref{eq:coercive gf}, we know that $|I_{\mathrm{out}}(\la)|\leq T_1+T_2+T_3$ for $p=1/(2L)$. Thus it suffices to upper bound $T_1$, $T_2$, $T_3$. We leave $p$ unspecified for now to show where the choice $p=1/(2L)$ comes from. Write $a=\kappa \la d^{-p}$. Using spherical coordinates and assuming $a>d$ gives
\be
T_1 \le
\frac{\la^{d/2}}{2^{(d/2)-1}\Gamma(d/2)}\int_0^\infty (1+r)^{-a}r^{d-1}\,\dd r
=
(\la/2)^{d/2}\,\frac{2\Gamma(d)}{\Gamma(d/2)}\cdot \frac{\Gamma(a-d)}{\Gamma(a)}.
\ee
We have 
\be
\frac{\Gamma(a-d)}{\Gamma(a)}
= \frac{1}{(a-d)(a-d+1)\cdots(a-1)}
\le
\Big(\frac{2}{a}\Big)^d,\quad
\frac{\Gamma(d)}{\Gamma(d/2)}\le (2d/e)^{d/2}.
\ee
We assumed $a\ge 2d$ in the first inequality and used the Stirling formula in the second one. Then
\be
T_1\leq
2\Big(\frac \la2\Big)^{d/2}\Big(\frac{2d}e\Big)^{d/2}\Big(\frac{2}{a}\Big)^d
=2\Big(\frac4e\frac{d^{2p+1}}{\kappa^2 \la}\Big)^{d/2}.
\ee
We now choose $p$ so that $d^{2p+1}/\la$ is small whenever the bound on $I_{\mathrm{in}}$ from Lemma~\ref{lma:in} is small. Thus we take $2p+1=(L+1)/L$, i.e. $p=1/(2L)$. This gives
\be\label{T1}
T_1 \leq 2\Big(\frac4e\frac{d^{(L+1)/L}}{\kappa^2 \la}\Big)^{d/2}.
\ee
Returning to the condition $a\geq 2d$, with $a=\kappa\la d^{-1/(2L)}$, this is satisfied if $d^{1+1/(2L)}/\la < \kappa/2$. But this is implied by the conditions $R^2\kappa>R\kappa>2$ and $R^{2(L+1)}d^{L+1}/\la^L<1/2$. These latter conditions are satisfied for appropriate choices of $c^\pm$ in~\eqref{R ineq v1}.

Next, using spherical coordinates to compute $T_2$, we have
\bs
T_2=&\frac2{2^{d/2}\Gamma(d/2)}\int_{Rd^{1/2}}^\infty e^{-\kappa d^{1/2}r}r^{d-1}\dd r
\leq  d^{1/2}e^{d/2}\int_R^\infty r^{d-1}e^{-\kappa dr}\dd r.
\es 
Let $f(r)=(d-1)\log r -\kappa dr$, so that $r^{d-1}e^{-\kappa dr}=e^{f(r)}$. Note that $f''(r)<0$ for all $r>0$, so $f'(r)$ is decreasing. Therefore, $f(r)\leq f(R)+f'(R)(r-R)$ for $r\geq R$. We have $e^{f(R)}=R^{d-1}e^{-\kappa dR}$ and $f'(R)=(d-1)/R-\kappa d$. We thus obtain
\bs\label{T2}
T_2 &\leq d^{1/2}e^{d/2}e^{f(R)}\int_R^\infty e^{f'(R)(r-R)}\dd r = \frac{d^{1/2}e^{d/2}}{-f'(R)}e^{f(R)}\\
& = \frac{d^{1/2}e^{d/2}R^de^{-\kappa dR}}{R\kappa d-(d-1)}\leq e^{-d[\kappa R-1/2-\log R]}.
\es
To get the last inequality we used $R\kappa d-(d-1)\geq d^{1/2}$, again by assuming $R\kappa\geq2$. 

Finally, using a Gaussian concentration inequality, we have
\bs\label{T3}
T_3&=\kappa^{-d}\frac{1}{(2\pi)^{d/2}}\int_{\|y\|\geq\kappa R\sqrt d}e^{-\|y\|^2/2}\dd y\leq \kappa^{-d}\exp\left(-\tfrac{d}{2}(\kappa R-1)^2\right)\\
&\leq\exp\left(d\left[\log\tfrac1\kappa-\tfrac18(\kappa R)^2\right]\right)\leq e^{-\frac{d}{16}(\kappa R)^2}.\es
To get the first inequality in the second line, we again used $R\kappa \geq2$. To get the second inequality in the second line, we used $(R\kappa)^2\geq 16\log\frac1\kappa$, which is satisfied by choosing $c^-\geq4$ in~\eqref{R ineq v1}.

Combining~\eqref{T1},~\eqref{T2}, and~\eqref{T3} gives
\bs\label{outub}
|I_{\mathrm{out}}(\la)| &\leq 2\Big(\frac4e\frac{d^{(L+1)/L}}{\kappa^2 \la}\Big)^{d/2} + e^{-d[\kappa R-1/2-\log R]}+e^{-\frac{d}{16}(\kappa R)^2}\\
&\leq 2\Big(\frac4e\frac{d^{(L+1)/L}}{\kappa^2 \la}\Big)^{d/2}+2e^{-d[\kappa R-1/2-\log R]}.
\es To get the second line we used that $\kappa R-1/2-\log R\leq \kappa R-1/2\leq (\kappa R)^2/16$, true for large enough $\kappa R$. We now finish the proof of~\eqref{out} using~\eqref{outub} and~\eqref{inlb}. 

\begin{proof}[Proof of~\eqref{out}]
~\eqref{outub} and~\eqref{inlb} give
\bs\label{Ioutin-0}
\frac{|I_{\mathrm{out}}(\la)|}{I_{\mathrm{in}}(\la)}&\leq 4\Big(\frac{4\mu}{e}\frac{d^{(L+1)/L}}{\kappa^2 \la}\Big)^{d/2} + 4e^{-d\big[\kappa R-\frac12-\frac12\log\mu-\log R\big]}.
\es 
Recall $\mu=1+\argc{f}R\sqrt\e$. By taking $c^+$ small enough in~\eqref{R ineq v1} we can ensure $\mu\leq e/2\leq e$. Thus we obtain the further bound
\bs\label{Ioutin}
\frac{|I_{\mathrm{out}}(\la)|}{I_{\mathrm{in}}(\la)}&\leq 4\Big(2\frac{d^{(L+1)/L}}{\kappa^2 \la}\Big)^{d/2} + 4\e^{-d[\kappa R-1-\log R]}\\
&\leq
4\Big(R^2\frac{d^{(L+1)/L}}{\la}\Big)^{L}+4\la^{-L}.
\es
To get the second inequality, we used that $2/\kappa^2\leq R^2$ and $R^2d^{(L+1)/L}/\la\leq 1$ by~\eqref{R ineq v1}, and that $d\geq2L$. We also used that $\kappa R-1-\log R\geq L(\log\la)/d$ by~\eqref{R ineq v1}, by choosing $c^-\geq L$. Finally, $R\ge 1$ and $d^{L+1}\ge4$ imply that the term $4\la^{-L}$ in \eqref{Ioutin} can be absorbed into the first term on the right-hand side of the last inequality, yielding
\be\label{Ioutin-1}
\frac{|I_{\mathrm{out}}(\la)|}{I_{\mathrm{in}}(\la)}\leq 5R^{2L}\frac{d^{L+1}}{\la^L}.
\ee
\end{proof}

\section{Sampling and expectations for Laplace-type densities }\label{sec:meas}
In statistical applications, the target of study is not a Laplace-type integral but a Laplace-type \emph{probability density},
\be
\pi(x)=\frac{e^{-\la f(x)}}{\int_{\br^d} e^{-\la f(x')}\dd x'}.
\ee 
Specifically, one is interested in obtaining the following quantities:
\begin{enumerate}
    \item i.i.d. samples $X_1,\dots,X_N\sim\pi$,
    \item Expectations $\E_{X\sim\pi}[g(X)]$.
\end{enumerate}
In Section~\ref{subsec:gen:meas} we present our general results for approximating these quantities to arbitrary order $L$. In Section~\ref{subsec:measL12} we specialize to the case $L=1$ and $L=2$.
\subsection{Results for general $L$}\label{subsec:gen:meas}
We start with the problem of computing expectations for \emph{smooth} functions $g$. We can do this using Theorem~\ref{thm:main res}. Indeed, we have $\E_{X\sim\pi}[g(X)]=\int ge^{-\la f}/\int e^{-\la f}$, and each integral can be approximated by~\eqref{main res}.
\begin{theorem}\label{thm:meas-smooth}
     Fix any $L\ge 1$ and suppose $d\geq 2L$. Suppose $f$ satisfies~\ref{Ial} of Assumption~\ref{ass:fg}, as well as~\eqref{f-suff} for some $\kappa_f>0$. Fix some $\argc{g}>0$ and $0<\kappa_g<\kappa_f$. Let $\mathcal G_L(\argc{g},\kappa_g)$ be the class of functions $g$ satisfying~\ref{Ial g} of Assumption~\ref{ass:fg} and~\eqref{g-suff}. Then for $R,d,\la$ as in Theorem~\ref{thm:main res}, with $\kappa=\kappa_f-\kappa_g$ in~\eqref{R ineq v1}, it holds
     \bs\label{supgG}
    \sup_{g\in\mathcal G_L(\argc{g},\kappa_g)}
    \bigg|\E_{X\sim\pi}[g(X)]-&\exp\bigg(\sum_{k=1}^{L-1}[b_k(f,g)-b_k(f,1)]/\la^k\bigg)\bigg|\leq C(\argc{g},\argc{f})\frac{(R^2d)^{L+1}}{\la^L}.
     \es 
 \end{theorem} 
 We have formulated the result uniformly over a function class $\mathcal G$ in order to compare it with our second result below about approximating $\E_{X\sim\pi}[g(X)]$ for nonsmooth functions $g$. Note that functions $g\in\mathcal G$ have $g(0)=1$. This is not restrictive; the theorem also applies for any $cg$, $g\in\mathcal G$, simply by multiplying~\eqref{supgG} through by $c$. 
 
 \begin{remark}[Highest derivative order]\label{rk:deriv-b}We claim that $\sum_{k=1}^{L-1}[b_k(f,g)-b_k(f,1)]/\la^k$ involves $g$ derivatives of order $\leq 2L-2$ and $f$ derivatives of order $\leq 2L-1$.
 
  By Lemma~\ref{prop:ad}, the terms $b_k(f,g)-b_k(f,1)$ for $k\leq L-2$ involve $g$ derivatives of order at most $2L-4$ and $f$ derivatives of order at most $2L-2$. Furthermore,~\eqref{hi-f-b} shows that the highest $f$ derivative appearing in $b_{L-1}(f,g)$ is $\E[\nabla^{2L}f(0)[Z^{\otimes 2L}]]/(2L)!$. But this exact term is also the highest $f$ derivative appearing in $b_{L-1}(f,1)$. Thus it cancels upon subtraction. As a result, $b_{L-1}(f,g)-b_{L-1}(f,1)$ only involves $g$ derivatives of order $\leq 2L-2$ and $f$ derivatives of order $\leq 2L-1$, and these are the highest derivative orders appearing in the sum.
 \end{remark}

 The proof of Theorem~\ref{thm:meas-smooth} is a straightforward application of Theorem~\ref{thm:main res}.
 \begin{proof}[Proof of Theorem~\ref{thm:meas-smooth}]
     First note that for $f$ as in the theorem statement and $g\in\mathcal G_L(\argc{g},\kappa_g)$, as well as $g\equiv1$, the conditions of Theorem~\ref{thm:main res} are satisfied. Let $I^g(\la)=(\la/2\pi)^{d/2}\int ge^{-\la f}$ and $I^1(\la)=(\la/2\pi)^{d/2}\int e^{-\la f}$. Also, let $b^g=\sum_{k=1}^{L-1}b_k(f,g)\la^{-k}$ and $b^1=\sum_{k=1}^{L-1}b_k(f,1)\la^{-k}$. We have
     \bs
\left|\frac{I^g(\la)}{I^1(\la)} - \frac{e^{b^g}}{e^{b^1}}\right| &= \frac{I^g(\la)}{I^1(\la)}\left|1-\frac{e^{b^g-\log I^g(\la)}}{e^{b^1-\log I^1(\la)}}\right| \\
&\leq \frac{I^g(\la)}{I^1(\la)}\left(e^{|b^g-\log I^g(\la)|+|b^1-\log I^1(\la)|}-1\right)\\
&\leq C(\argc{g},\argc{f})\frac{(R^2d)^{L+1}}{\la^L}\frac{I^g(\la)}{I^1(\la)}.
     \es Here, we have applied~\eqref{main res} and assumed by~\eqref{R ineq v1} that $(R^2d)^{L+1}/\la^L$ is small enough. We also used that $I^g(\la)>0$ by Theorem~\ref{thm:main res}. We have $I^g(\la)/I^1(\la) \leq \max_{\|x\|\leq r_0}|g(x)| +|I^g_{\mathrm{out}}(\la)|/I^1_{\mathrm{in}}(\la)$. The first summand is bounded by a function of $\argc{g}$ using~\ref{Ial g} of Assumption~\ref{ass:fg}, and the second term is bounded by an absolute constant, e.g. 1, using the arguments in Section~\ref{sec:out}.
 \end{proof}
Next, we construct an approximation $\hat\pi_L$ of $\pi$ which is easy to sample from, and which can be used to approximate $\E_{X\sim\pi}[g(X)]$ for nonsmooth $g$. The following theorem is our second main result.
\begin{theorem}\label{thm:meas} Fix any $L\ge 1$ and suppose $d\geq 2L$. Suppose $f$ satisfies~\ref{Ial} of Assumption~\ref{ass:fg} and~\eqref{f-suff}. Let $\kappa=\kappa_f$ and suppose $R,d,\la$ satisfy~\eqref{R ineq v1} for large enough $c^+=c^+(\argc{f})$ and small enough $c^-=c^-(\argc{f})$. Then it holds
\be\label{main res meas}
\mathrm{TV}(\pi,\hat\pi_L)=\tfrac12\sup_{\|g\|_\infty\leq1}\left|\E_{X\sim\pi}[g(X)]-\E_{X\sim\hat\pi_L}[g(X)]\right|\leq C(\argc{f})\frac{(R^2d)^{L+1}}{\la^L},
\ee where $\hat\pi_L=(x_L)_{\#}\mathcal N(0, \la^{-1}I_d)$. 
The map $x_L$ is defined by
\[
x_L \;=\; T_0\circ T_1\circ\cdots\circ T_{L-1}\circ Q .
\]
Here $T_0=X$ is as in Lemma~\ref{lma:changevar}, while the remaining maps are those obtained by carrying out the
constructions of Section~\ref{sec:in} with $g\equiv 1$ (i.e., for the Laplace integral with integrand
$e^{-\la f(x)}$). In particular, $T_m$, $m=1,\dots,L-1$ are the maps from Lemma~\ref{lma:changevar-m}
and $Q(s)=B^{-1/2}(s-B^{-1/2}a)$ is defined using $a,B$ from Definition~\ref{def:aB}, all in the $g\equiv1$ setting.
\end{theorem}

The maps in the composition increase in complexity as one goes outward: $Q$ is linear, and $T_{L-m}$ is a polynomial of degree $2m$, $m=1,\dots, L$.
\begin{remark}
    Although we cannot prove~\eqref{main res meas} via Theorem~\ref{thm:main res} since $g$ is not smooth, nearly all the proof ingredients from Theorem~\ref{thm:main res} can be reused. 
 \end{remark}
 \begin{remark}\label{rk:xLderiv}
     Constructing $\hat\pi_L$ requires computing the first $2L+1$ derivatives of $f$. To see this, recall from Lemma~\ref{lma:changevar} that $T_0=X$ is a change of variables ensuring that $f_{2L+1}(T_0(t))=\|t\|^2/2+\CO(\|t\|^{2L+2})$, where $f_{2L+1}(x)=\sum_{k=3}^{2L+1}\frac{1}{k!}\nabla^kf(0)[x^{\otimes k}]$. Thus clearly, $T_0$ should depend on $\nabla^kf(0)$, $k=3,\dots,2L+1$. Since $\hat\pi_L=(x_L)_{\#}\mathcal N(0, \la^{-1}I_d)$ and $x_L$ is a composition involving $T_0$, constructing $\hat\pi_L$ also requires these derivatives.
 \end{remark}
 Theorem~\ref{thm:meas} indeed gives a tractable algorithm for approximately sampling from $\pi$: simply draw $Z_i\sim\mathcal N(0, \la^{-1}I_d)$ i.i.d. and return $x_L(Z_i)$. To do this, we push $Z_i$ through the sequence of $L+1$ maps $Q, T_{L-1},T_{L-2},\dots,T_0$. In the case $L=1$, we explicitly construct $x_1$ in Section~\ref{subsec:measL12}. 

See also further discussion of the uses for Theorems~\ref{thm:meas-smooth} and~\ref{thm:meas} in Section~\ref{subsec:meas:discuss}.

\begin{proof}[Proof of Theorem~\ref{thm:meas}]
    In this proof, $\les$ suppresses a constant depending only on $\argc{f}$. We need to prove that
    \be\label{toprove}\left|\textstyle\int gd\pi -\E\left[g\left(x_L(Z_\la)\right)\right]\right|\les R^{2L+2}d\e^L\ee for all $\|g\|_\infty\leq1$, where $Z_\la\sim\mathcal N(0, \la^{-1}I_d)$. By writing $g=\max(g,0)+\min(g,0)$ and applying triangle inequality in~\eqref{toprove}, it further suffices to only consider functions $g\geq0$, $\|g\|_\infty\leq1$. Fix such a $g$. Write $I^g_{\mathrm{in}}(\la)$, $I^g_{\mathrm{out}}(\la)$ instead of $I_{\mathrm{in}}(\la)$, $I_{\mathrm{out}}(\la)$, respectively. We start by bounding $I^g_{\mathrm{out}}(\la)/I^1_{\mathrm{in}}(\la)$ using essentially the exact same technique as in the proof of Lemma~\ref{lma:out} in Section~\ref{sec:out}. Note that $ge^{-\la f}$ satisfies~\eqref{eq:coercive gf} with $\argc{g}=0$ because $f$ satisfies~\eqref{f-suff} and $g$ is bounded by 1. Therefore, the upper bound on $I^g_{\mathrm{out}}(\la)$ from~\eqref{outub} remains true. Regarding conditions on $R,\kappa,d,\la$, the proof of~\eqref{outub} uses~\eqref{R ineq v1} only with absolute constants $c^\pm$.

    Furthermore, the lower bound on $I^1_{\mathrm{in}}(\la)$ from~\eqref{inlb} is also true, and it requires only that $R\sqrt\e\leq r_0$ and $d\geq2L\geq2$. 
    
    Applying~\eqref{outub} to upper bound $|I^g_{\mathrm{out}}(\la)|$ and~\eqref{inlb} to lower bound $I^1_{\mathrm{in}}(\la)$, we conclude that the ratio $|I^g_{\mathrm{out}}(\la)
    |/I^1_{\mathrm{in}}(\la)$ satisfies the exact same upper bound as in~\eqref{Ioutin-0}.  We can still conclude the final inequality~\eqref{Ioutin-1} by choosing $c^-$ and $c^+$ that depend on $\argc{f}$ only. Thus
    \be
    |I^g_{\mathrm{out}}(\la)
    |/I^1_{\mathrm{in}}(\la) \leq 5(R^2d)^{L+1}/\la^L=5R^{2L+2}d\e^L
    \ee and by the same logic, 
    \be
    I^1_{\mathrm{out}}(\la)
    /I^1_{\mathrm{in}}(\la) \leq 5R^{2L+2}d\e^L.
    \ee Now, note that $\int gd\pi = I^g(\la)/I^1(\la)$. Omitting the argument $(\la)$ for brevity, we then have
    \be\label{intgxL}
|\textstyle\int gd\pi - \E[g(x_L(Z_\la))]|\leq |I^g/I^1-I^g_{\mathrm{in}}
    /I^1_{\mathrm{in}}| + |I^g_{\mathrm{in}}
    /I^1_{\mathrm{in}}-\E[g(x_L(Z_\la))]|.
    \ee
    Furthermore,
    \be\label{intgxL2}
|I^g/I^1-I^g_{\mathrm{in}}
    /I^1_{\mathrm{in}}| \leq \frac{|I^g-I^g_{\mathrm{in}}|}{I^1}+\frac{|I^g_{\mathrm{in}}|}{I^1_{\mathrm{in}}}\left|\frac{I^1_{\mathrm{in}}}{I^1}-1\right| \leq \frac{|I^g_{\mathrm{out}}|}{I^1_{\mathrm{in}}}+\frac{I^1_{\mathrm{out}}}{I^1_{\mathrm{in}}}\les R^{2L+2}d\e^L.
    \ee To get the second inequality, we used $\frac{|I^g_{\mathrm{in}}|}{I^1_{\mathrm{in}}}\leq1$ since $\|g\|_\infty\leq1$. It remains to study the term $|I^g_{\mathrm{in}}
    /I^1_{\mathrm{in}}-\E[g(x_L(Z_\la))]|$ from~\eqref{intgxL}. We revise the argument in Section~\ref{sec:in}. Since the change of variables $T_0(t):=X(t)$ from Lemma~\ref{lma:changevar} does not depend on $g$, we only need $R\sqrt\e$ smaller than a constant depending on $\argc{f}$ alone in the argument below Lemma~\ref{lma:U}. We conclude, analogously to~\eqref{initchange} but without bringing $g$ into the exponent, that
    \be\label{Igin}
    I^g_{\mathrm{in}}(\la)=e^{\CO(R^{2L+2}d\e^L)}\left(\frac{\la}{2\pi}\right)^{d/2}\int_{\us_1}g(T_0(t))\exp\left(-\la\left[\tfrac12\|t\|^2-\tfrac1\la\log\det(T_0'(t))\right]\right)\dd t.\ee For~\eqref{Igin} and other multiplicative big-$\CO$ identities below to be valid, we use that $g\geq0$. Next, we expand $\log\det(T_0'(t))$ as in Lemma~\ref{lma:logdet exp}. The lemma goes through unchanged, except that the $\ef{k}$'s do not depend on any derivatives of $g$. (Recall from Definition~\ref{def:base comp} that $\ef{k}$'s are in principle allowed to depend on both derivatives of $f$ and $g$.) Combining Lemma~\ref{lma:logdet exp} with~\eqref{Igin} gives
\bs\label{E1-no-g}
&I_{\mathrm{in}}^g(\la)=e^{\CO(R^{2L+2}d\e^L)}\left(\frac{\la}{2\pi}\right)^{d/2}\int_{\us_1}g(T_0(t))\exp\left(-\la E_1(t)\right)\dd t,\\
&E_1(t):=\tfrac12\norm{t}^2+\e\sum_{k=1}^{2L-1}\ef{k}[t^{\otimes k}],\\
&\{\|t\|\leq \tfrac23R\sqrt\e\}\subseteq\us_1\subseteq\{\|t\|\leq 2R\sqrt\e\}.
\es 
Note that this $E_1(t)$ is precisely what we would get in~\eqref{E1} if we had taken $g\equiv1$ in the argument in Section~\ref{sec:in}. We now iterate from $E_m$ to $E_{m+1}$ exactly as in Lemmas~\ref{lma:changevar-m} and~\ref{lma:phi-m}. We then conclude~\eqref{Em1}. Combining~\eqref{Em1} with~\eqref{E1-no-g} (when $m=1$, and then iteratively updating~\eqref{E1-no-g}) we conclude the following analogue of~\eqref{Emm1}:
\bs\label{Emm1-variant}
\int_{\us_m}g((T_0\circ&\dots\circ T_{m-1})(t))e^{-\la E_m(t)}\dd t \\
&= e^{\CO(R^{2L+2}d\e^L)}\int_{\us_{m+1}}g((T_0\circ\dots\circ T_{m})(s))e^{-\la E_{m+1}(s)}\dd s.
\es
Thus, starting with~\eqref{E1-no-g} and applying~\eqref{Emm1-variant} with $m=1,\dots, L-1$, we conclude that
\bs\label{EL-no-g}
I^g_{\mathrm{in}}(\la)=e^{\CO(R^{2L+2}d\e^L)}\left(\frac{\la}{2\pi}\right)^{d/2}\int_{\us_L}g(T_0\circ T_1\circ\dots\circ T_{L-1}(t))\exp\left(-\la E_L(t)\right)\dd t.
\es
Finally, we complete the square as in~\eqref{after sqr}, and use the change of variables $t=Q(s)$ from~\eqref{tQs}. This gives, analogously to~\eqref{eBa}, that
\bs\label{Igfinal}
I^g_{\mathrm{in}}(\la)=&e^{\CO(R^{2L+2}d\e^L)}e^{\frac\la2a^\top B^{-1}a}(\det B)^{-1/2}\\
&\times\left(\frac{\la}{2\pi}\right)^{d/2}\int_{\tilde\us_L}g(T_0\circ\dots\circ T_{L-1}\circ Q(s))\exp\left(-\la\|s\|^2/2\right)\dd s\\
=&e^{\CO(R^{2L+2}d\e^L)}e^{\frac{\la}{2}a^\top B^{-1}a}(\det B)^{-1/2}\E\left[g(x_L(Z_\la))\ind\{Z_\la\in\tilde\us_L\}\right],
\es
where $\tilde\us_L=B^{1/2}\us_{L}+B^{-1/2}a$. By the same logic and for the same $ B,a,\tilde\us_L$, we have
\be\label{I1final}
I_{\mathrm{in}}^1(\la)=e^{\CO(R^{2L+2}d\e^L)} e^{\frac{\la}{2}a^\top B^{-1}a}(\det B)^{-1/2}P(Z_\la\in\tilde\us_L).
\ee 
Now, note that $P(Z_\la\in\tilde\us_L)=P(Z\in\sqrt\la\tilde\us_L)$ for $Z\sim\mathcal N(0, I_d)$ and consider the proof of Lemma~\ref{lma:square}. The same logic can be used to show $P(Z\in\sqrt\la\tilde\us_L)\geq1-\la^{-L}$ as in~\eqref{PZlb}, with the one modification that all constants need only depend on $\argc{f}$, not on $\argc{g}$. Combining this lower bound with~\eqref{I1final} and~\eqref{Igfinal} gives
\be
\frac{I_{\mathrm{in}}^g(\la)}{I_{\mathrm{in}}^1(\la)}=e^{\CO(R^{2L+2}d\e^L)}\E\left[g(x_L(Z_\la))\ind\{Z_\la\in\tilde\us_L\}\right].
\ee Since $\|g\|_\infty\leq 1$ and using the lower bound on $\P(Z_\la\in\tilde\us_L)$, we have $$\E\left[g(x_L(Z_\la))\ind\{Z_\la\in\tilde\us_L\}\right]=\E\left[g(x_L(Z_\la))\right]+\CO(\la^{-L}).$$ Therefore,
$$
\frac{I_{\mathrm{in}}^g(\la)}{I_{\mathrm{in}}^1(\la)}=e^{\CO(R^{2L+2}d\e^L)}\E\left[g(x_L(Z_\la))\right]+\CO(\la^{-L}),
$$ so that
\be\label{intgxL3}
\left|\frac{I_{\mathrm{in}}^g(\la)}{I_{\mathrm{in}}^1(\la)}-\E\left[g(x_L(Z_\la))\right]\right|\les R^{2L+2}d\e^L.
\ee Combining~\eqref{intgxL},~\eqref{intgxL2}, and~\eqref{intgxL3} finishes the proof of~\eqref{toprove}.
\end{proof}

\subsection{Special cases: $L=1$ and $L=2$}\label{subsec:measL12}
We derive the approximation from Theorem~\ref{thm:meas-smooth} in the cases $L=1$, $L=2$, and the approximation from Theorem~\ref{thm:meas} in the case $L=1$.  

Theorem~\ref{thm:meas-smooth} with $L=1$ gives that for all $g\in\mathcal G_1(\argc{g},\kappa_g)$, we have
\be
\E_{X\sim\pi}[g(X)]=1+\CO\left(R^4\frac{d^2}{\la}\right),
\ee where $\CO$ suppresses dependence on $\argc{g}$ and $\argc{f}$. Here, recall that we assume $g(0)=1$. To work out the case $L=2$, we compute $b_1(f,g)-b_1(f,1)$. Recall the formula for $b_1(f,g)$ from~\eqref{b1fg}. When we subtract $b_1(f,1)$, all terms involving only $f$ will cancel. The remaining expression is $b_1(f,g)-b_1(f,1)=-\frac12\nabla\Delta f(0)^\top\nabla g(0)+\frac12\Delta g(0)$. We conclude that for all $g\in\mathcal G_2(\argc{g},\kappa_g)$, we have 
$$\E_{X\sim\pi}[g(X)]=\exp\left(-\frac1{2\la}\nabla\Delta f(0)^\top\nabla g(0)+\frac1{2\la}\Delta g(0)\right)+\CO\left(R^6\frac{d^3}{\la^2}\right).$$

Next, we compute the map $x_1=T_0\circ Q$ from Theorem~\ref{thm:meas}. The map $T_0$ is $T_0=X$ from Lemma~\ref{lma:changevar}. Below this lemma, we showed that when $L=1$ we have \be\label{T0-1}T_0(t)=t-\frac16\nabla^3f(0)[t^{\otimes2}].\ee To determine $Q$, we need to derive $E_1(t)$ and write it in the form~\eqref{ELaB}. The function $E_1$ is given by adding the nonnegligible part of $\frac1\la\log\det T_0'(t)$ to $\|t\|^2/2$. We have
$$
\frac1\la\log\det T_0'(t)=\frac1\la\tr\log\left(I_d-\tfrac1{3}\nabla^3f(0)[t]\right)=-\tfrac1{3\la}\tr(\nabla^3f(0)[t])+\e\CO(\|t\|^2),
$$ as in Lemma~\ref{lma:logdet exp} but explicitly computing that $d\ef{1}[t]=-\tfrac13\tr(\nabla^3f(0)[t])$. Here, $\nabla^3f(0)[t]$ is the matrix with $(i,j)$th entry given by $\nabla^3f(0)[t, e_i, e_j]$. Throwing out $\e\CO(\|t\|^2)$, we conclude
\be\label{E1-1}E_1(t)=\frac12\|t\|^2 + \frac1{3\la}\tr(\nabla^3f(0)[t])=\frac12\|t\|^2+\frac1{3\la}\nabla\Delta f(0)^\top t.\ee 
Comparing with~~\eqref{ELaB}, we see that $a=\frac1{3\la}\nabla\Delta f(0)$ and $B=I_d$. Thus Theorem~\ref{thm:meas} gives 
\be\label{S-1}
Q(s) = B^{-1/2}(s-B^{-1/2}a)=s-\frac1{3\la}\nabla\Delta f(0).\ee
Finally, $x_1=T_0\circ Q$. Thus Theorem~\ref{thm:meas} with $L=1$ gives that
\bs\label{x1-g}
\E_{X\sim\pi}[g(X)] &= \E\left[g\left(S-\frac16\nabla^3f(0)[S,S,\cdot]\right)\right] + \CO\left(R^4\frac{d^2}{\la}\right),\\
S&\sim\mathcal N\left(-\frac1{3\la}\nabla\Delta f(0), \la^{-1}I_d\right)
\es for all $\|g\|_\infty\leq1$, where $\CO$ suppresses dependence on $\argc{f}$ only. Here, we have used that $\hat\pi_1=(T_0\circ Q)_{\#}\mathcal N(0,\la^{-1}I_d)$, which is the pushforward under $T_0$ of $Q_{\#}\mathcal N(0,\la^{-1}I_d)=\mathcal N(-\frac1{3\la}\nabla\Delta f(0), \la^{-1}I_d)$. 

The righthand expectation in~\eqref{x1-g} typically cannot be evaluated in closed-form, but can easily be approximated by Monte Carlo using samples $S_i\sim \mathcal N(-\frac1{3\la}\nabla\Delta f(0), \la^{-1}I_d)$.

\section{Comparison to the literature and examples}\label{sec:examples}
In Section~\ref{subsec:compareintegral}, we compare our integral expansion result to that of~\cite{katsevich2025a}. Then in Sections~\ref{quartic} and~\ref{logreg}, we derive the integral expansion for a few examples. In Section~\ref{subsec:meas:discuss}, we discuss our results involving Laplace-type densities from Section~\ref{sec:meas}, and compare them to related work in the literature. 

Throughout the section, we assume $d\geq\log\la$ to simplify formulas. 
\subsection{Comparison to integral expansions in the literature}\label{subsec:compareintegral}
In the below informal discussion of our results, we neglect any dependence on $r_0,\kappa$. Since also $d\geq\log\la$, a constant $R$ can be used to satisfy~\eqref{R ineq v1}. Theorem~\ref{thm:main res} then shows that
\be\label{curr}
I(\la)=\left(\frac{\la}{2\pi}\right)^{d/2}\int_{\br^d}g(x)e^{-\la f(x)}\dd x = \exp\left(\sum_{k=1}^{L-1}b_k\la^{-k} + \CO(d^{L+1}/\la^L)\right),\quad |b_k|\les d^{k+1},
\ee 
provided $\max_{1\leq k\leq 2L}\|\nabla^k(\log g)(x)\|_\op < \argc{g}$ and $\max_{3\leq k\leq 2L+2}\|\nabla^kf(x)\|_\op < \argc{f}$, uniformly over $x$ in a small neighborhood of 0. As discussed in the introduction, in~\cite{katsevich2025a} it is shown that
\be\label{prev}
I(\la)=\left(\frac{\la}{2\pi}\right)^{d/2}\int_{\br^d}g(x)e^{-\la f(x)}\dd x =\sum_{k=1}^{L-1}a_k\la^{-k} + \mathcal O(d^{2L}/\la^L),\quad |a_k|\les d^{2k},
\ee provided $\|\nabla^kg(x)\|_\op\les d^{\lceil k/2\rceil}$, $k=1,\dots,2L$ and $\|\nabla^kf(x)\|_\op\les d^{\lceil k/2\rceil-2}$, $k=3,\dots,2L+2$, uniformly over $x$ in a small neighborhood of 0. The suppressed constants in the big-$\CO$ and in the bound on $|a_k|$ in~\eqref{prev} depend on the suppressed constants in the derivative bounds.

Both results also require a few other minor assumptions, but we have highlighted the most important ones for this discussion. 
We make a few comments on the difference between the two works.
\begin{enumerate}
\item For~\eqref{curr}, we have assumed the derivative operator norms are bounded independently of $d$. This was not required in~\cite{katsevich2025a}. However, if the assumption does hold, then our result is strictly stronger than that of~\cite{katsevich2025a}. Specifically, suppose the operator norms are bounded and $d^2\ll\la$. Then~\eqref{prev} can be derived from~\eqref{curr}.

\item As can be seen from~\cite{katsevich2025a}, $d^2\ll \la$ is necessary for the expansion of $I(\la)$ even when the derivative operator norms are bounded.  Indeed, consider the example in Section~\ref{quartic}, also studied in Example 2.19 in~\cite{katsevich2025a}. We see that the derivative operator norms are indeed bounded, yet~\cite{katsevich2025a} proves that the expansion of $I(\la)$ is valid only if $d^2\ll \la$. This shows that our improved dimension dependence \emph{cannot simply be attributed to the stricter requirement we have imposed on the operator norms}. Rather, it is due to the intrinsic difference between expanding $I(\la)$ and expanding $\log I(\la)$, as described in the introduction.
\item In the expansion of $I(\la)$ in~\cite{katsevich2025a}, the derivative operator norms beyond the fourth order were allowed to grow with $d$ due to slack in the bound. A similar phenomenon may hold for the expansion of $\log I(\la)$. Namely, there may be some slack which would permit derivative norm growth with $d$. We leave this investigation to future work.
\end{enumerate}


\subsection{Example: quartic exponent}\label{quartic}Let $f(x)=\|x\|^2/2 + \|x\|^4/24$ and $g(x)\equiv1$. Consider~\eqref{eq:Taylor-f} and~\eqref{Fk norms}. For all $L\geq2$ we have $\mathcal R_{2L+2}(x)\equiv0$, and $\nabla^3f(0)[x^{\otimes3}]=0$, $\nabla^4f(0)[x^{\otimes4}]=\|x\|^4$, $\nabla^kf(0)[x^{\otimes k}]=0$ for all $k\geq5$. Thus~\eqref{eq:Taylor-f} and~\eqref{Fk norms} are satisfied for any $r_0$. Also, $\log g\equiv 0$ trivially satisfies~\eqref{eq:Taylor-g},~\eqref{Gk norms} for any $r_0$. Furthermore,~\eqref{eq:coercive gf} holds with $\kappa=1$, since $|g(x)|e^{-\la f(x)}\leq e^{-\la\|x\|^2/2}$. 

Theorem~\ref{thm:main res} therefore applies. Since we assume $d\geq\log\la$ throughout the section, we can take $R$ to be an absolute constant in~\eqref{R ineq v1}. We conclude that
\bs\label{quartL}
\log\left\{\left(\frac{\la}{2\pi}\right)^{d/2}\int_{\br^d}e^{-\la \left(\frac{\|x\|^2}{2}+\frac{\|x\|^4}{24}\right)}\dd x\right\}  =\sum_{k=1}^{L-1} b_k\la^{-k}+\CO\left(d^{L+1}/\la^L\right),
\es for all $d^{L+1}/\la^L$ small enough. Here, the constant suppressed by the big-$\CO$ depends on $L$ only. Next, let us compute $b_1$, given by~\eqref{b1fg}. Plugging in $f(x)=\|x\|^2/2 + \|x\|^4/24$ and $g(x)\equiv1$ gives
\bs\label{b1fg-quart}
b_1 &=-\frac18\Delta^2f(0)=-\frac18\frac{1}{24}\sum_{i,j=1}^d\pa_i^2\pa_j^2(\|x\|^4)\\
&= -\frac18\frac{1}{24}\sum_{i\neq j}\pa_i^2\pa_j^2(2x_i^2x_j^2)-\frac18\frac{1}{24}\sum_{i=1}^d\pa_i^4(x_i^4)\\
&=-\frac1{24}(d^2-d) -\frac18d = -\frac1{24}d^2 -\frac{1}{12}d.
\es
Using this in~\eqref{quartL} with $L=2$, we conclude that
\be\label{I1}
I(\la)=\left(\frac{\la}{2\pi}\right)^{d/2}\int_{\br^d}e^{-\la \left(\frac{\|x\|^2}{2}+\frac{\|x\|^4}{24}\right)}\dd x =  \exp\left(-\frac{d^2}{24\la} -\frac{d}{12\la} + \CO(d^3/\la^2)\right),
\ee where the constant suppressed by the big-$\CO$ is absolute. 

This example is also studied in~\cite{katsevich2025a}, where it is shown that $I(\la)= 1-(d^2/24+d/12)/\la + \CO(d^4/\la^2)$, but \emph{only if $d^2/\la\ll1$}. Our result improves on~\cite{katsevich2025a} in that it both tightens the remainder bound and broadens the range of applicability of the expansion into higher $d$ regimes.



\subsection{Example: logistic regression}\label{logreg}
Next, we consider an idealized statistical set-up, in which $\int_{\br^d}e^{-\la f(x)}\dd x$ is the normalizing constant of the posterior in a logistic regression model. See Section~\ref{ssec:statistics_motivation} for more details on how this quantity arises in statistics. We call the large parameter $\la=n$ the sample size. Let $X_1,\dots, X_n\in\br^d$ be the feature column vectors, and $x^*\in\br^d$ be the ground truth we would like to estimate. 
We assume (1) $X_1,\dots,X_n$ span $\br^d$, (2) the lowest eigenvalue of $\frac1n\sum_{i=1}^nX_iX_i^\top$ is bounded below away from zero by some $c_X>0$, and (3) $\max_{i=1,\dots,n}\|X_i\|\max(1, \|x^*\|)<C_X$. Next, let $\psi:\br\to\br$ be a smooth strictly convex function, with $\|\psi^{(k)}\|_\infty=\sup_{t\in\br}|\psi^{(k)}(t)|<\infty$ for all $k=2,3,4,\dots$. Consider the function
$$\ell(x) = \frac1n\sum_{i=1}^n\left[\psi(X_i^\top x)-\psi'(X_i^\top x^*)X_i^\top x\right]$$ which has a unique global minimizer at $x=x^*$. When $\psi(t)=\log(1+e^{t})$, the function $\ell$ is the negative, normalized population log likelihood for the logistic regression model. (The fact that $\ell$ is the \emph{population} rather than the \emph{sample} log likelihood makes this an idealized set-up.) 

Let
\be
H:=\nabla^2\ell(x^*)=\frac1n\sum_{i=1}^n\psi''(X_i^\top x^*)X_iX_i^\top.
\ee
The matrix $H$ is strictly positive definite. Indeed, let $c_\psi=\inf_{|t|<C_X}\psi''(t)$, which is positive since $\psi''(t)>0$ for all $t\in\br^d$ and $\psi$ is smooth. Then $H\succeq c_\psi\,\frac1n\sum_{i=1}^nX_iX_i^\top\succeq  c_\psi c_XI_d$. 
By a change of variables, we have
\be\label{ell-f}\int_{\br^d}e^{-n\ell(x)}\dd x =\frac{e^{-n\ell(x^*)}}{\sqrt{\det H}}\int_{\br^d}e^{-nf(x)}\dd x,\qquad f(x) = \ell(H^{-1/2}x+x^*)-\ell(x^*).
\ee 
We verify the conditions of Theorem~\ref{thm:main res} for this $f$. We have $f(0)=0$, $\nabla f(0)=0$ and $\nabla^2f(0)=I_d$. Thus the Taylor expansion of $f$ indeed takes the form in~\eqref{eq:Taylor-f}. Furthermore, using the above lower bound on $H$ and the definition of $\ell$, we have for $k\ge 3$
\bs
\sup_{x\in\br^d}\|\nabla^kf(x)\|_\op&\leq (c_\psi c_X)^{-k/2}\sup_{x\in\br^d}\|\nabla^k\ell(x)\|_\op \\
&= (c_\psi c_X)^{-k/2}\sup_{x\in\br^d}\left\|\frac1n\sum_{i=1}^n\psi^{(k)}(X_i^\top x)X_i^{\otimes k}\right\|_\op\\
&\leq (c_\psi c_X)^{-k/2}\|\psi^{(k)}\|_\infty\max_{i=1,\dots,n}\|X_i\|^k\\
&\leq (c_\psi c_X)^{-k/2}\|\psi^{(k)}\|_\infty C_X^k.
\es
Thus~\eqref{Fk norms} is satisfied for any $r_0$. The condition on $g\equiv1$ in Assumption~\ref{ass:fg}, part~\ref{Ial g} is trivially satisfied. Finally, to verify~\eqref{eq:coercive gf}, we prove the sufficient condition~\eqref{f-suff} holds. Even more simply, it suffices to show $f(x)\geq\kappa\min(\|x\|^2/2, \sqrt{\e}\|x\|)$ for all $x\in\br^d$. Let $c=\max_{u\in\br^d}\|\nabla^3f(x)\|_\op$, which we know is bounded, by the above calculation. Then for all $\|x\|\leq\sqrt\e$, a Taylor expansion gives $f(x)\geq (1-c\sqrt\e/3)\|x\|^2/2 \geq\frac12(\|x\|^2/2)$, if $\e$ is small enough. For all $\|x\|\geq\sqrt\e$, convexity of $f$ implies that $f(x)\geq \frac{\|x\|}{\|y\|}f(y)$, where $\|y\|=\sqrt\e$ and $y$ lies on the line segment between $0$ and $x$. But now we use that $f(y)\geq\|y\|^2/4$ since $\|y\|=\sqrt\e$, so $f(x)\geq\|x\|\|y\|/4=\|x\|\sqrt\e/4$. Therefore, $f(x)\geq\kappa\min(\|x\|^2/2, \sqrt{\e}\|x\|)$ is satisfied on $\br^d$, with $\kappa=1/4$.

Theorem~\ref{thm:main res} therefore applies to $f$ from~\eqref{ell-f}. As in Section~\ref{quartic}, we can satisfy the lower bound in~\eqref{R ineq v1} by taking $R=c$ for some  constant $c$ depending on $C_X$, $c_x$, $\psi$, and $L$. Writing
$$
\int_{\br^d}e^{-n\ell(x)}\dd x= \frac{e^{-n\ell(x^*)}}{\sqrt{\det H}}\left(\frac{2\pi}{n}\right)^{d/2}\times \left\{\left(\frac{n}{2\pi}\right)^{d/2}\int_{\br^d}e^{-nf(x)}\dd x\right\},
$$ we conclude that
\bs\label{logreg-L}
\log\int_{\br^d}e^{-n\ell(x)}\dd x = &-n\ell(x^*)-\frac12\log\det H -\frac d2\log\frac{n}{2\pi}\\
&+\frac{b_1}{n}+\CO\left( \frac{d^3}{n^2}\right),
\es whenever $d^3/n^2$ is sufficiently small. Here, the constant suppressed by the big-$\CO$ depends on $C_X,c_x,\psi$. The term $b_1$ can be computed using the general formula~\eqref{b1fg}, and similar calculations as in~\cite{katsevich2025a,katsevich2024b}. Note that the righthand side of the first line in~\eqref{logreg-L} is the Bayesian information criterion (BIC)~\cite{neath2012bayesian}. Thus $b_1/n$ can be considered a higher-order correction to BIC, with the higher-order remainder explicitly controlled.

\subsection{Comparison to Laplace-type density approximations in the literature}\label{subsec:meas:discuss}
Here, we discuss the problem of approximating Laplace-type densities.
We first compare the two approximation strategies from Theorems~\ref{thm:meas-smooth} and~\ref{thm:meas}, showing that the former uses fewer $f$ derivatives at the cost of requiring smoothness of $g$, while the latter offers sampling flexibility. We then compare the combination of these two methods to other approaches in the literature. Our discussion in this section is focused on the statistical context.

Recall from Section~\ref{sec:meas} that Theorem~\ref{thm:meas-smooth} gives the approximation
 \be\label{Eg-expand}
\E_{X\sim \pi}[g(X)]= \exp\left(\sum_{k=1}^{L-1}[b_k(f,g)-b_k(f,1)]\la^{-k}\right) + \CO(d^{L+1}/\la^L),
 \ee 
 while Theorem~\ref{thm:meas} gives
 \be\label{Eg-expect}
\E_{X\sim \pi}[g(X)]= \E_{X\sim \hat\pi_L}[g(X)]+ \CO(d^{L+1}/\la^L),\qquad \hat\pi_L:=(x_L)_{\#}\mathcal N(0, \la^{-1}I_d).
 \ee  
 Here as above, we assume $d\geq\log\la$ and neglect dependence on $r_0,\kappa$. This allows us to take $R$ to be constant, which is why $R$ does not appear in the big-$\CO$'s above. 

 When~\eqref{Eg-expand} is applicable, it is a more powerful method for computing expectations than~\eqref{Eg-expect}. The reason for this is two-fold. First,~\eqref{Eg-expect} cannot be implemented exactly, and requires a further sampling step, as seen in~\eqref{eq:monte-carlo-approx}. This incurs additional computational cost and loss of accuracy. On top of this, using $\hat\pi_L$ requires computing higher-order $f$ derivatives than are involved in $\sum_{k=1}^{L-1}[b_k(f,g)-b_k(f,1)]/\la^k$. Indeed, Remarks~\ref{rk:xLderiv} and~\ref{rk:deriv-b} show that $\hat{\pi}_L$ uses $2L+1$ derivatives of $f$ derivatives, while $\sum_{k=1}^{L-1}[b_k(f,g)-b_k(f,1)]/\la^k$ uses only $2L-1$. The reason for this gap is that the two approximations achieve the same accuracy, but $\hat\pi_L$ does not exploit any smoothness of $g$; it must compensate by extracting more information from $f$. In practice, this difference in derivative count can be significant for computational cost. Note that the closed-form approximation~\eqref{Eg-expand} does require derivatives of $g$, while the pushforward approximation~\eqref{Eg-expect} does not. However, in Bayesian statistics $g$ is typically a simple function (e.g. linear or quadratic, corresponding to the mean or covariance of $\pi$), making its derivatives cheap to compute. In contrast, $f$ encodes the data likelihood and involves a sum over $\la=n$ terms, so each additional derivative of $f$ carries substantial computational cost.

The number of derivatives of $f$ and $g$ used in the approximations~\eqref{Eg-expand} and~\eqref{Eg-expect} are summarized in Table~\ref{table} for the cases $L=1$ and $L=2$.
\begin{table}[ht]
\centering
\begin{tabular}{l|cc|cc}
 & \multicolumn{2}{c|}{$L=1$} & \multicolumn{2}{c}{$L=2$} \\
 & closed-form \eqref{Eg-expand} & pushforward \eqref{Eg-expect}
 & closed-form \eqref{Eg-expand} & pushforward \eqref{Eg-expect} \\
\hline
formula & \eqref{Eg1} & \eqref{Epi1} & \eqref{Eg2} & --- \\
$\#$ derivs of $f$ & 0 & 3 & 3 & 5 \\
$\#$ derivs of $g$ & 0 & 0 & 2 & 0 \\
\end{tabular}
\caption{Number of derivatives used in the approximation of $\E_{X \sim\pi}[g(X)]$ by the closed form estimate~\eqref{Eg-expand} (first and third columns) and the pushforward estimate~\eqref{Eg-expect} (second and fourth columns). The two approximations for $L=1$ both have accuracy $\CO(d^2/\la)$, and the two approximations for $L=2$ both have accuracy $\CO(d^3/\la^2)$.}
\label{table}
\end{table}
For convenience, we remind the reader of the explicit formulas for three of the approximations referenced in the table: 
\begin{alignat}{2}
&\E_{X\sim\pi}[g(X)]\approx 1,\qquad &g\in \mathcal G_1(\argc{g},\kappa_g),\label{Eg1}\\
&\E_{X\sim\pi}[g(X)] \approx \E\left[g\left(S-\tfrac16\nabla^3f(0)[S,S,\cdot]\right)\right] ,\qquad &\|g\|_\infty\leq1,\label{Epi1}\\
&S\sim\mathcal N\left(-\tfrac1{3\la}\nabla\Delta f(0), \la^{-1}I_d\right),&\notag\\
&\E_{X\sim\pi}[g(X)]\approx\exp\left(-\tfrac1{2\la}\nabla\Delta f(0)^\top\nabla g(0)+\tfrac1{2\la}\Delta g(0)\right),\qquad &g\in \mathcal G_2(\argc{g},\kappa_g).\label{Eg2}
\end{alignat}
These were derived in Section~\ref{subsec:measL12}. We have omitted the formula for the approximation via $\E_{X\sim\pi_2}[g(X)]$ (i.e. the fourth column), which requires 5 derivatives of $f$. 

Although the density approximations $\hat\pi_L$ to $\pi$ are more derivative-intensive, they can be used to approximate expectations of nonsmooth functions $g$, which the closed-form method~\eqref{Eg-expand} cannot do. Furthermore, approximately sampling from $\pi$ is itself valuable, beyond just computing expectations. For example, in Bayesian statistics, samples can be used to construct approximate credible intervals to quantify uncertainty in the target parameter of inference. The algorithm to sample from $\hat\pi_1$ is especially simple:
\begin{enumerate}
    \item Draw $S_i\sim\mathcal N(-\frac1{3\la}\nabla\Delta f(0),\la^{-1}I_d)$ i.i.d.
    \item Return $X_i=S_i-\frac16\nabla^3f(0)[S_i,S_i,\cdot]$.
 \end{enumerate}
The combined ability to 1) easily generate approximate samples from $\pi$ (via the above algorithm), 2) approximate expectations of nonsmooth $g$ (via the Monte Carlo estimate), and 3) accurately and cheaply approximate expectations for smooth $g$ (via~\eqref{Eg1} and~\eqref{Eg2}) is extremely powerful compared to the state of the art.

Two noteworthy alternative methods in the literature for high-accuracy sampling and computing expectations are~\cite{durante2024skewed} and~\cite{katskew}. The two works construct approximations $\hat P_{\mathrm{SKS}}$ and $\hat\gamma_S$ to $\pi$, respectively. They are similar to our $\hat\pi_1$, involving only the second and third derivatives of $f$. (The reason our $\hat\pi_1$ only involves the third derivative is that we have assumed $\nabla^2f(0)=I_d$. For a generic $\nabla^2f(0)$, the second derivative will also arise.) 

We argue that our combined approach takes the best elements of each of the approximation methods in the above works. For sampling accuracy, the relevant metric is TV distance. In ~\cite{durante2024skewed}, the authors show only that $\mathrm{TV}(\pi,\hat P_{\mathrm{SKS}})\les d^3/\la$, whereas we show the tighter dimension dependence $\mathrm{TV}(\pi,\hat\pi_1)\les d^2/\la$. Furthermore, while $\hat P_{\mathrm{SKS}}$ is easy to sample from, it cannot be integrated against in closed-form. As a result, \emph{Monte Carlo sampling is always needed} for the purpose of computing expectations, even of smooth functions. As discussed above, this incurs extra computational cost as well as an additional source of error. But even if it were possible to compute $\E_{X\sim \hat P_{\mathrm{SKS}}}[g(X)]$ exactly, the accuracy of the approximation remains $d^3/\la$ at best. In contrast, by exploiting the smoothness of $g$, our estimate~\eqref{Eg2} achieves the much higher accuracy $d^3/\la^2$ while \emph{using the same number of derivatives of $f$}. This is a significant improvement by a factor of $1/\la$.

The approximation $\hat\gamma_S$ of~\cite{katskew} is a signed measure, not a true probability density. It has the advantage that expectations of polynomials against $\hat\gamma_S$ can be computed in closed-form, unlike $\hat P_{\mathrm{SKS}}$. However, our~\eqref{Eg-expect} gives a closed-form approximation for expectations of $g$ in the even broader class of smooth functions, not just polynomials.
Also, the fact that $\hat\gamma_S$ is not a true probability density has disadvantages; for example, ``sampling" from this signed measure is not well-defined. Our $\hat\pi_1$ does not have this issue: it is a true probability density and can be easily sampled from. 

Another significant advantage of our results, compared to those of~\cite{katskew} and~\cite{durante2024skewed}, is that we give a method to approximate $\pi$ (and expectations under $\pi$) to arbitrary order of accuracy. In contrast, the other two works focus only on a fixed order of approximation, and it is unclear whether their constructions (or proof techniques) can be extended to higher orders of accuracy. 

Finally, it is natural to compare~\eqref{Eg-expand} to the analogous result from~\cite{katsevich2025a}, which can be used to expand the Laplace integral in the numerator and denominator as follows:
\be
\E_{X\sim\pi}[g(X)]=\frac{\sum_{k=0}^{L-1}c_k(f,g)\la^{-k}+\CO((d^2/\la)^L)}{\sum_{k=0}^{L-1}c_k(f,1)\la^{-k}+\CO((d^2/\la)^L)}.
\ee
But the drawback of this result, as already discussed in Section~\ref{subsec:compareintegral}, is that it does not allow $d$ to be larger than $\sqrt\la$.

In summary, our approach combines the best features of these methods --- closed-form expectations for smooth $g$, easy sampling from a true density --- and adds two more features: arbitrary-order accuracy, and validity up to the concentration threshold.

\bibliographystyle{plain} 
\bibliography{bibliogr_Lap_expansion, MyLibrary}
\appendix

\section{Structure-preserving tensor operations}\label{sec:tensor ops}
Recall from Section~\ref{sec:notation} the notation of a tensor $A_{k\to j}$: a multilinear mapping from $k$ vectors in $\br^d$ to either a scalar if $j=0$, a vector in $\br^d$ if $j=1$, and a $d\times d$ matrix if $j=2$. Recall also the concept of a base tensor $G_{k\to j}$ and a composite tensor $\ef{k\to j}$.

In the next lemma, we list some operations which preserve the structure of a composite tensor. Recall that all the composite tensors $\ef{k\to j}$ and all the base tensors $G_{k\to j}$ they are composed of are symmetric.

\begin{lemma}\label{lma:op}
We have the following identities:
\begin{enumerate}[label=(\arabic*)]
\item ($\e$-scaling) $\e^p\ef{k\to j}=\ef{k\to j}$ if $p\geq0$.
\item\label{comp} (composition) $\ef{n\to j}\left[\ef{k_1\to1}[x^{\otimes k_1}],\dots, \ef{k_n\to1}[x^{\otimes k_n}]\right]=\ef{K\to j}[x^{\otimes K}]$, where $K=k_1+\dots+k_n$.
\item (matrix multiplication) $\ef{k_1\to 2}[x^{\otimes k_1}]\ef{k_2\to 2}[x^{\otimes k_2}]\cdots\ef{k_\ell\to 2}[x^{\otimes k_\ell}]=\ef{K\to2}[x^{\otimes K}]$, where $K=k_1+\dots+k_\ell$ and $AB$ refers to matrix multiplication of $A$ and $B$.
\item (trace) $\frac1d\tr(\ef{k\to2}[x^{\otimes k}])=\ef{k\to0}[x^{\otimes k}]$.
\end{enumerate}
In each case, the equality should be read as follows: given the composite tensors appearing on each lefthand side, there exists a composite tensors of the form given on the righthand side to make the equality true.
\end{lemma}
\begin{remark}\label{rk:matmul}We will have use for the third identity with $k_1=\dots=k_\ell=0$. The identity then gives that the product of matrices of the form $\sum_\ell\e^\ell G_{0\to2}^{(\ell)}$ is also such a matrix.
\end{remark}

\begin{proof}
The first identity is trivial: clearly the structure is preserved, and boundedness of operator norms is unaffected. To prove the second identity, multilinearity gives that the lefthand side is a sum of nonnegative powers of $\e$ times terms of the form
\be\label{Gn}
G_{n\to j}\Big[G_{k_1\to1}[x^{\otimes k_1}],\dots, G_{k_n\to1}[x^{\otimes k_n}]\Big].
\ee 
Define $T_{K\to j}[x_1,\dots,x_K]$ by
 \bs\label{TKj}
 T_{K\to j}[x_1,\dots, x_K]=\frac{1}{K!}\sum_\sigma G_{n\to j}\Big[&G_{k_1\to1}[x_{\sigma(1)},\dots,x_{\sigma(k_1)}], G_{k_2\to1}[x_{\sigma(k_1+1)},\dots,x_{\sigma(k_1+k_2)}],\\
 &\dots, G_{k_n\to1}[x_{\sigma(K-k_n+1)},\dots,x_{\sigma(K)}]\Big],\es
where the sum is over all permutations $\sigma$ of $\{1,\dots,K\}$. Then $T_{K\to j}$ satisfies
\be
T_{K\to j}[x^{\otimes K}]=G_{n\to j}\left[G_{k_1\to1}[x^{\otimes k_1}],\dots, G_{k_n\to1}[x^{\otimes k_n}]\right].
\ee
 It remains to show $T_{K\to j}$ is a base tensor. It is symmetric by construction, and only depends on $d,\nabla^mf(0),\nabla^m\log g(0)$ since this is true for each of $G_{n\to j},G_{k_1\to1},\dots, G_{k_n\to1}$. Furthermore, we have $\|T_{K\to j}\|_\op \leq \|G_{n\to j}\|_\op\|G_{k_1\to1}\|_\op\dots\|G_{k_n\to1}\|_\op\leq c$.
 
To  prove the third identity, multilinearity gives that the lefthand side is a sum of nonnegative powers of $\e$ times terms of the form
$$G_{k_1\to 2}[x^{\otimes k_1}] G_{k_2\to 2}[x^{\otimes k_2}]\cdots G_{k_\ell\to2}[x^{\otimes k_\ell}]$$ for base tensors $G_{k_1\to2},\dots,G_{k_\ell\to2}$. As in~\eqref{TKj}, define a symmetric tensor $T_{K\to2}$ such that $T_{K\to2}[x^{\otimes K}]=G_{k_1\to 2}[x^{\otimes k_1}] G_{k_2\to 2}[x^{\otimes k_2}]\cdots G_{k_\ell\to2}[x^{\otimes k_\ell}]$. (Recall that symmetry of $T_{K\to2}$ is with respect to the input arguments. The output matrix need not be symmetric.)  It remains to show $T_{K\to2}$ is a base tensor. It is symmetric by construction, and only depends on $d,\nabla^mf(0),\nabla^m\log g(0)$ since this is true for each $G_{k_1\to2},\dots,G_{k_\ell\to2}$. Furthermore, we have \be
\|T_{K\to 2}\|_\op \leq \|G_{k_1\to2}\|_\op\dots\|G_{k_\ell\to2}\|_\op < c. 
\ee

The fourth identity is straightforward.
\end{proof}

\begin{corollary}\label{corr:pq} Let $\mathcal K$ be a finite subset of $\mathbb N\cup\{0\}$. It holds
\be\label{njouter}
\ef{n\to j}\left[\left(\sum_{k\in\mathcal K}\ef{k\to1}[x^{\otimes k}]\right)^{\otimes n}\right]=\sum_{q\in\mathcal Q}\ef{q\to j}[x^{\otimes q}],
\ee and
\be\label{k2power}
\bigg(\sum_{k\in\mathcal K}\ef{k\to2}[x^{\otimes k}]\bigg)^{n} = \sum_{q\in\mathcal Q}\ef{q\to 2}[x^{\otimes q}],
\ee 
for some finite subset $\mathcal Q\subset\mathbb N\cup\{0\}$, 
where $\min\{q:\,q\in\mathcal Q\}=n\cdot \min\{k:\,k\in\mathcal K\}$.
\end{corollary}

\begin{proof}~\eqref{njouter} follows from multilinearity of the $^{\otimes n}$ operation and the second identity of Lemma~\ref{lma:op}. Similarly,~\eqref{k2power} follows from multilinearity of the matrix power operation, and the third identity of Lemma~\ref{lma:op}.
\end{proof}


\begin{lemma}\label{lma:logdet} Let $A(t)=\sum_{k\geq1}\ef{k\to2}[t^{\otimes k}]$, $m\geq 0$, and $r\leq1$ be small enough. Then there are composite tensors $\ef{k}$ such that for any $N\ge2$, we have
\bs
\log\det(I_d+\e^mA(t))=d\e^m\left(\sum_{k=1}^{N-1}\ef{k\to0}[t^{\otimes k}]+\CO\big(\|t\|^N\big)\right),\qquad\forall \|t\|\leq r.
\es
\end{lemma}

\begin{proof}
As is well known, $\log\detm(I_d+\e^mA(t))=\tr\log(I_d+\e^mA(t))$. Due to the form of $A(t)$ and the assumption $r\leq1$, we have $\|\e^mA(t)\|_\op \leq\|A(t)\|_\op \les \|t\|$ for all $\|t\|\leq r$. 
We assume $r$ is sufficiently small that $\|\e^mA(t)\|_\op \leq 1/2$ for all $\|t\|\leq r$. We then have
\bs\label{basic logexp}
\Bigl\|\log(I_d+\e^mA(t))&-\e^m\bigg\{\sum_{k=1}^{N-1}\frac{(-1)^{k+1}}{k}\,\e^{m(k-1)}A(t)^k\bigg\}\Bigr\|_{\op}\\
&\le \sum_{k=N}^{\infty}\frac{\|\e^mA(t)\|^k}{k}\les\|\e^mA(t)\|^N\les \e^{mN}\|t\|^N,\quad\forall \|t\|\leq r.
\es 
By the first identity in Lemma~\ref{lma:op}, and~\eqref{k2power} in Corollary~\ref{corr:pq}, the sum in curly braces can be expressed as $\sum_{\ell\geq1}\ef{\ell\to2}[t^{\otimes\ell}]$. Furthermore, we have $\|\sum_{\ell\geq N}\ef{\ell\to2}[t^{\otimes\ell}] \|\les\|t\|^N$ for all $\|t\|\leq r$. Therefore,
\be
\Bigl\|\log(I_d+\e^mA(t))-\e^m\sum_{\ell=1}^{N-1}\ef{\ell\to2}[t^{\otimes\ell}]\Bigr\|_\op \les\e^m\|t\|^N\quad\forall \|t\|\leq r.
\ee
Using that $|\tr A-\tr B|\leq d\|A-B\|$, and using the fourth identity in Lemma~\ref{lma:op} concludes the proof.
\end{proof}

\begin{lemma}\label{lma:circ}Let $\ef{k_1\to1},\ef{k_2\to1},\dots, \ef{k_\ell\to1}$ be composite tensors, with $k_i\geq2$ for all $i$, and let $t_i(s)=s+\e^m\ef{k_i\to1}[s^{\otimes k_i}]$, $i=1,\dots, \ell$, where $m\geq0$. Then there exist $\ef{p\to1}$, $p\geq2$, such that 
\be\label{tell}(t_1\circ t_2\circ\dots\circ t_\ell)(s)=s+\e^m\sum_{p\geq2}\ef{p\to1}[s^{\otimes p}].\ee 
\end{lemma}
\begin{proof}We use induction. The result trivially holds for $\ell=1$. Suppose~\eqref{tell} holds for some $\ell-1\geq1$. Then
\bs
(t_1\circ t_2\circ\dots\circ t_{\ell-1})(t_{\ell}(s))&=t_{\ell}(s)+\e^m\sum_{p\geq2}\ef{p\to1}[t_{\ell}(s)^{\otimes p}]\\
&=s+\e^m\bigg(\ef{k_\ell\to1}[s^{\otimes k_\ell}]+\sum_{p\geq2}\ef{p\to1}\left[\left(s+\e^m\ef{k_\ell\to1}[s^{\otimes k_\ell}]\right)^{\otimes p}\right]\bigg)\\
&=s+\e^m\sum_{q\geq2}\ef{q\to1}[s^{\otimes q}].
\es The last line (including that $q\geq2$) is by the assumption $k_\ell\ge 2$ and by~\eqref{njouter} of Corollary~\ref{corr:pq}.
\end{proof}

\section{Change of variable proofs}\label{app:change}
\begin{proof}[Proof of Lemma~\ref{lma:U}] Recall that $\varphi(0)=0$. The assumption
\be\label{phid}
\|\varphi'(t)\|_\op \leq 1/2,\quad\forall\|t\| \leq 2C_2r,
\ee 
implies that $\|\varphi(t)\|\leq \|t\|/2$, $\|t\|\leq 2C_2r$. This gives 
\be\label{twosided incl}
\|(\mathrm{id}+\varphi)(t)\| \leq \tfrac32\|t\| \leq C_1r,\quad\forall \|t\|\leq \tfrac23C_1r.
\ee
We conclude $\{\|t\|\leq \tfrac23 C_1 r\}\subset X^{-1}(\us)$. 

Since $\us\subset \{\|x\|\leq C_2r\}$, we will finish the proof of (1) and prove (2) by showing that for any $\|x\|\leq C_2r$ there exists a unique $\|t\|\leq 2C_2r$ such that $X(t)=x$. Fix any such $x$ and define $\varphi_x(t)=x-\varphi(t)$. We have
\be
\varphi_x(\{\|t\|\leq 2C_2r\})\subset \{\|u\|\leq 2C_2r\},
\ee
because $\|t\|\leq 2C_2r$ and \eqref{phid} imply
\bs
\|\varphi_x(t)\|&\leq\|x\|+\|\varphi(t)-\varphi(0)\|\leq C_2r +\tfrac12\|t\| \leq 2C_2r.
\es 
Furthermore, again using~\eqref{phid}, it is straightforward to show that $\varphi_x$ is a strict contraction on $\{\|t\|\leq 2C_2r\}$. Therefore, the contraction mapping theorem shows there is a unique $\|t\|\leq 2C_2r$ satisfying $\varphi_x(t)=t$. But then $x=(\mathrm{id}+\varphi)(t)$, proving the claims. 
\end{proof}

For the proof of Lemma~\ref{lma:changevar}, we state and prove a main auxiliary lemma. Recall that $\ef{k}$ is shorthand for $\ef{k\to0}$.
\begin{lemma}\label{lma:M}
Let $3\leq M\leq 2L+1$ and $f$ be a function on $\br^d$ given by
\be\label{fxM}
f(x)=\frac12\|x\|^2 +\sum_{k\geq M}\ef{k}[x^{\otimes k}]
\ee Then there exists a composite tensor $\ef{(M-1)\to1}$ such that
\be\label{fxMx}
f\left(x- \ef{(M-1)\to1}[x^{\otimes M-1}]\right)=\frac12\|x\|^2 +\sum_{k\geq M+1}\ef{k}[x^{\otimes k}].
\ee 
\end{lemma}
\begin{proof}
Let $\ef{M}$ be the specific composite tensor appearing in~\eqref{fxM}. For each $x$, let $A_{(M-1)\to1}[x^{\otimes M-1}]$ be the vector such that $A_{(M-1)\to1}[x^{\otimes M-1}]^\top u=\ef{M}[x^{\otimes M-1}, u]$ for all $u$. The order of the $M$ arguments input to $\ef{M}[\cdot]$ is irrelevant, since $\ef{M}$ is symmetric by definition. It is straightforward to see that $A_{(M-1)\to1}$ is composite, so from now on we call it $\ef{(M-1)\to1}$. Let $p(x)=- \ef{(M-1)\to1}[x^{\otimes M-1}]$. Then $x^\top p(x) = -\ef{M}[x^{\otimes M}]$, by definition of $\ef{(M-1)\to1}$. Thus
$$\tfrac12\|x+p(x)\|^2 =\tfrac12\|x\|^2- \ef{M}[x^{\otimes M}] + (\tfrac12I_d)[p(x)^{\otimes 2}].$$ Here, we are identifying the matrix $I_d$ with the bilinear form $I_d[u\otimes u]=\sum_{i,j=1}^d(I_d)_{ij}u_iu_j = \|u\|^2$. We then have
\bs\label{fxp}
f(x+p(x))=&\frac12\|x\|^2 +\bigg\{\ef{M}[(x+p(x))^{\otimes M}-x^{\otimes M}]\\
&+ (\tfrac12I_d)[p(x)^{\otimes 2}] + \sum_{k\geq M+1}\ef{k}[(x+p(x))^{\otimes k}]\bigg\}.
\es 
We expand the outer products in the terms inside the curly braces. By Corollary~\ref{corr:pq}, the result of doing these outer product expansions is a sum of the form $\sum_q\ef{q}[x^{\otimes q}]$. It suffices to show $q\geq M+1$ for all $q$ in the sum. 

For $(\tfrac12I_d)[p(x)^{\otimes 2}]$, we have $q=2M-2\geq M+1$, since $M\geq3$. For $\ef{k}[(x+p(x))^{\otimes k}]$, $k\geq M+1$, all resulting $\ef{q}[x^{\otimes q}]$ have $q\geq M+1$. Finally, expanding the outer product, $\ef{M}[(x+p(x))^{\otimes M}-x^{\otimes M}]$ is a sum of terms of the form $\ef{M}[x^{\otimes M-m}\otimes p(x)^{\otimes m}]=\ef{q}[x^{\otimes q}]$ (by~\eqref{njouter} with $j=0$) for $q=(M-m)+(M-1)m = M+(M-2)m$, with $m=1,\dots, M$. Since $m\geq1$ we have $q\geq2M-2\geq M+1$.
\end{proof}
\begin{proof}[Proof of Lemma~\ref{lma:changevar}]
Note that $f_{2L+1}$ satisfies the conditions of Lemma~\ref{lma:M} with $M=3$. In fact, the tensors in the expansion of $f_{2L+1}$ are base tensors, which are a special case of composite tensors. We iteratively apply the lemma, with $M=3,4,\dots, 2L+1$, to get that
$(f_{2L+1}\circ x_3\circ x_4\circ\dots\circ x_{2L+1})(t)=\frac12\|t\|^2 +\sum_{k\geq 2L+2}\ef{k}[t^{\otimes k}]$.  But now, assuming $2R\sqrt\e\le1$, we have that $\sum_{k\geq 2L+2}\ef{k}[t^{\otimes k}]=\CO(\|t\|^{2L+2})$ for all $\|t\|\leq2R\sqrt\e$. Thus $f(h(t))=\|t\|^2/2+\CO(\|t\|^{2L+2})$ for all $\|t\|\leq 2R\sqrt\e$, where 
$h=x_3\circ x_4\circ\dots\circ x_{2L+1}$. We now modify $h$. Each $x_k$ is of the form $x_k(t)=t+\ef{k\to1}[t^{\otimes k}]$, and $k\geq2$. Therefore, Lemma~\ref{lma:circ} with $m=0$ gives
\be\label{h-def}
h(t)=t+\sum_{k\geq2}\ef{k\to1}[s^{\otimes k}].\ee Next, define $X(t)=t+\sum_{k=2}^{2L}\ef{k\to1}[s^{\otimes k}]$ for the same $\ef{k\to1}$ as in~\eqref{h-def}. Thus $h(t)=X(t)+\CO(\|t\|^{2L+1})$ for all $\|t\|\leq1$. But then, using that $\|X(t)\|=\CO(\|t\|)$, $\|t\|\leq1$, we have
\bs
|f_{2L+1}(h(t))-f_{2L+1}(X(t))|&\leq\sum_{k=2}^{2L+1}\frac1{k!}\left|\nabla^kf(0)[h(t)^{\otimes k}-X(t)^{\otimes k}]\right|\\
&\les\sum_{k=2}^{2L+1}\|t\|^{2L+1}\|t\|^{k-1} = \CO(\|t\|^{2L+2}).
\es Therefore, $f_{2L+1}(X(t))=\|t\|^2/2+\CO(\|t\|^{2L+2})$ as well, and $X(t)$ is the desired polynomial change of variables of order $2L$.

\end{proof}

For the proof of Lemma~\ref{lma:changevar-m}, we state and prove a main auxiliary lemma.
\begin{lemma}\label{lma:M-m}
Let $m\geq1$, $M\geq 3$ and
\bs\label{eq:M-m}
f(x)=&\e \ef{1}[x]+\e \ef{2}[x^{\otimes 2}] + \frac12\|x\|^2 \\
&+\e^{m+1}\sum_{k=3}^{M-1}\ef{k}[x^{\otimes k}]+ \e^{m}\sum_{k\geq M}\ef{k}[x^{\otimes k}].
\es Then there exists a composite tensor $\ef{(M-1)\to1}$ such that
\bs\label{tildFm}
f\left(x- \e^m\ef{(M-1)\to1}[x^{\otimes M-1}]\right)=&\e \ef{1}[x]+\e \ef{2}[x^{\otimes 2}] + \frac12\|x\|^2 \\
&+\e^{m+1}\sum_{k=3}^{M}\ef{k}[x^{\otimes k}]+ \e^{m}\sum_{k\geq M+1}\ef{k}[x^{\otimes k}].
\es  
\end{lemma} 

\begin{proof}
Let $\ef{M}$ be the specific composite tensor arising in~\eqref{eq:M-m}. As in the proof of Lemma~\ref{lma:M}, we construct the composite tensor $\ef{(M-1)\to1}$ such that $\ef{(M-1)\to1}[x^{\otimes M-1}]^\top u=\ef{M}[x^{\otimes M-1}, u]$ for all $x,u\in\br^d$. Let $p(x)=- \e^m\ef{(M-1)\to1}[x^{\otimes M-1}]$. Note that $$\tfrac12\|x+p(x)\|^2 =\tfrac12\|x\|^2-\e^m \ef{M}[x^{\otimes M}] + \tfrac12\e^{2m}I_d[p(x)^{\otimes2}].$$ Then $f(x+p(x))=\frac12\|x\|^2+A(x)$, where
\bs\label{fxp-m}
A(x):=&\e \sum_{k=1}^2\ef{k}[(x+p(x))^{\otimes k}]+  \tfrac12\e^{2m}I_d[p(x)^{\otimes2}]+\e^m\ef{M}[(x+p(x))^{\otimes M}-x^{\otimes M}] \\
&+\e^{m+1}\sum_{k=3}^{M-1}\ef{k}[(x+p(x))^{\otimes k}] + \e^{m}\sum_{k\geq M+1}\ef{k}[(x+p(x))^{\otimes k}].
\es 
We now study $A(x)$. Let us expand the outer products, but not yet collect terms by like powers of $x$. This gives $A(x)=\sum_{n,\ell}\e^{\ell}\ef{n}[x^{\otimes n}]$, by Lemma~\ref{lma:op} and Corollary~\ref{corr:pq}. To prove~\eqref{tildFm}, it now suffices to show
\begin{itemize}
\item $\ell\geq1$ for $n=1,2$, 
\item $\ell\geq m+1$ for all $n=3,\dots, M$.
\item $\ell\geq m$ for all $n\geq M+1$,
\end{itemize}
We go through each term that arises in the expansion of $A(x)$. Whenever $\ell\geq m+1$, all three of the above cases are automatically satisfied, so we don't need to check what $n$ is.

\begin{itemize}
\item We have $\e \ef{1}[x+p(x)]=\e\ef{1}[x]-\e^{m+1}\ef{1}[\ef{(M-1)\to1}[x^{\otimes M-1}]]$. In the first term, we have $n=1$ and $\ell=1$. In the second term, we have $\ell=m+1$. 
\item The first term in the expansion of $\e \ef{2}[(x+p(x))^{\otimes 2}]$ has $n=2$, $\ell=1$. The second and third terms have $\ell\geq m+1$. 
\item For $\tfrac12\e^{2m}I_d[p(x)^{\otimes2}]$, we have $\ell=2m\geq m+1$ because $m\ge1$. 
\item The terms arising when $\e^m\ef{M}[(x+p(x))^{\otimes M}-x^{\otimes M}]$ is expanded each have $\ell\geq m$. For $n$, we have $n=(M-q)+(M-1)q$, where $q=1,\dots, M$. Thus $n=M+(M-2)q\geq 2M-2\geq M+1$.  
\item The terms arising from the sum $\e^{m+1}\sum_{k=3}^{M-1}\ef{k}[(x+p(x))^{\otimes k}]$ have $\ell\geq m+1$. 
\item The terms arising from $\e^{m}\sum_{k=M+1}^{N-2m+1}\ef{k}[(x+p(x))^{\otimes k}]$ have $\ell\geq m$, and $n\geq M+1$. 
\end{itemize}
\end{proof}

\begin{proof}[Proof of Lemma~\ref{lma:changevar-m}]
Note that $E_m$ satisfies the conditions of Lemma~\ref{lma:M-m} with $M=3$ (with the first sum on the second line in \eqref{eq:M-m} omitted). We iteratively apply the lemma, with $M=3,4,\dots, 2L-2m+1$, to get that
\bs
(E_m\circ h)(s)=&\e\ef{1}[s] + \e\ef{2}[s^{\otimes 2}]+\tfrac12\norm{s}^2+\e^{m+1}\sum_{k=3}^{2L-2m+1}\ef{k}[s^{\otimes k}]\\
&+\e^m\sum_{k\geq2L-2m+2}\ef{k}[s^{\otimes k}].
\es Here, $h(s)=t_3\circ t_4\circ\dots\circ t_{2L-2m+1}$, and each $t_k$ is of the form $t_k(s)=s+\e^m\ef{q_k\to1}[t^{\otimes q_k}]$, with $q_k\geq2$.

Next, if $\|s\|\leq CR\sqrt\e\leq1$ and $k\geq 2L-2m+2$, then $\e^m|\ef{k}[s^{\otimes k}]|\les \e^m(R\sqrt\e)^{2L-2m+2} =\CO(R^{2L+2}\e^{L+1})$. Furthermore,
\be
\e^{m+1}\left|\sum_{k=2L-2m}^{2L-2m+1}\ef{k}[s^{\otimes k}]\right|=\CO(R^{2L+2}\e^{L+1})
\ee as well. Thus we obtain
\be\label{Emt-h}
 E_m(h(s))=\e\ef{1}[s] + \e\ef{2}[s^{\otimes 2}]+\tfrac12\norm{s}^2+\e^{m+1}\sum_{k=3}^{2L-2m-1}\ef{k}[s^{\otimes k}]+\CO\big((R\sqrt\e)^{2L+2}\big)
 \ee  for all $\|s\|\leq CR\sqrt\e$. Next, Lemma~\ref{lma:circ} gives that $h(s)=s+\e^m\sum_{k\geq2}\ef{k\to1}[s^{\otimes k}]$. Define $T_m(s)=s+\e^m\sum_{k=2}^{2L-2m}\ef{k\to1}[s^{\otimes k}]$ for the same tensors $\ef{k\to1}$ as in $h$. Thus $h(s)-T_m(s)=\e^m\CO(\|s\|^{2L-2m+1})$ for all $\|s\|\leq CR\sqrt\e$. This and the fact that $\|T_m(s)\|=\CO(\|s\|)$ imply $$|\ef{k}[h(s)^{\otimes k}-T_m(s)^{\otimes k}]|\les \sum_{j=1}^{k}\e^{mj}\CO(\|s\|^{(2L-2m+1)j})\CO(\|s\|^{k-j})=\e^m\CO(\|s\|^{2L-2m+k})$$ for any $k\geq1$. Recalling from~\eqref{Em} the definition of $E_m$, it follows that
\bs\label{Emht}
 |E_m(h(s))-E_m(T_m(s))|&\les\e^{1+m}\|s\|^{2L-2m+1}+\e^m\|s\|^{2L-2m+2}+\e^{2m}\|s\|^{2L-2m+3}\\
 &\les R^{2L-2m+1}\e^{L+1}.
\es Combining~\eqref{Emt-h} and~\eqref{Emht} concludes the proof.

\end{proof}

\end{document}